\documentclass[oneside,reqno]{amsart}
\usepackage{amsfonts}
\usepackage{amsmath}
\usepackage{amssymb}
\usepackage{graphicx}
\usepackage{version}
\usepackage{bm}
\usepackage{enumerate}
\usepackage{caption}%
\usepackage{setspace}
\newcommand{\R}{{\mathbb R}}

\theoremstyle{plain}
\newtheorem{theorem}{Theorem}[section]
\newtheorem{corollary}[theorem]{Corollary}
\newtheorem{lemma}[theorem]{Lemma}
\newtheorem{proposition}[theorem]{Proposition}
\theoremstyle{definition}
\newtheorem{definition}[theorem]{Definition}
\newtheorem{example}[theorem]{Example}

\newtheorem{remark}[theorem]{Remark}

\newtheorem*{notation*}{Notation}
\numberwithin{equation}{section}

\excludeversion{comment1}

\begin{document}
%\title{Iterated Transformations, Trajectories, Fractals and Subdivision}
%\title{Sequences of Function Systems and their Attractors}
\title{Attractors of Sequences of Function Systems \\ and their relation to Non-Stationary Subdivision}
\author[David Levin]{David Levin}
\author[Nira Dyn]{Nira Dyn}
\address{D. Levin, N. Dyn, School of Mathematical Sciences, Tel Aviv University, Israel}
\author[P. V. Viswanathan]{Puthan Veedu Viswanathan}
\address{P. V. Viswanathan, Department of Mathematics, Indian Institute of Technology, Delhi, India}

%\maketitle
\begin{abstract}
Iterated Function Systems (IFSs) have been at the heart of fractal geometry almost from
its origin, and several generalizations for the notion of IFS have been suggested.
Subdivision schemes are widely used in computer graphics and attempts have been made to link fractals generated by IFSs to limits
generated by subdivision schemes. With an eye towards establishing connection between non-stationary subdivision schemes and
fractals,  this paper introduces the notion of ``trajectories of maps defined by function systems" which may be considered as a new generalization of the traditional IFS. The significance and the convergence properties of `forward' and `backward' trajectories are studied.
Unlike the ordinary fractals which are self-similar at different scales, the attractors of these trajectories may have different structures at different scales.

\end{abstract}
\maketitle

\section{\bf Introduction}\label{sect1}

The concept of Iterated Function system (IFS) was introduced by Hutchinson  \cite{H} and popularized
by Barnsley \cite{B1}. IFSs form
a standard framework for describing self-referential sets such as fractals  and provide a potential new method of researching the shape and texture of images.
Due to its importance in understanding images, several extensions to the classical IFS such as Recurrent IFS, partitioned IFS and Super IFS are discussed in the literature \cite{B2,BEH,F}. Fractal functions whose graphs are attractors of suitably chosen IFS provide a new method of interpolation and approximation \cite{B1,PRM,MAN,PV1}.

\par

Subdivision schemes are efficient algorithmic methods for generating curves and surfaces from discrete sets of control
points. A subdivision scheme generates values associated with the vertices of a sequence of nested meshes,
by repeated application of a set of local refinement rules. These subdivision rules, usually linear, iteratively transform the vertices of a given mesh to vertices of a refined mesh. In recent years, the subject of subdivision has gained more popularity because of many new applications such as computer graphics. The reader may turn to \cite{CDM,DL,MP,PBP} for an introduction and survey of the mathematics of subdivision schemes and their applications.

\par

Being two different topics that had been developing independently  and in parallel, the connections  between subdivision and theory of IFS were sought after.
Later it has been observed that there is a close connection between
curves and surfaces generated by subdivision algorithms and self-similar fractals generated by IFSs \cite{SLG}. However, this relationship is established for stationary subdivision schemes. The relation between non-stationary subdivision and IFS remains obscure and unexplored.

\par
In this paper we target to establish the interconnection between the theory of IFS and non-stationary subdivision schemes. In this attempt, we introduce and study what we call  "trajectories of a sequence of transformations". Trajectories generated by a sequence of function system maps may provide new attractor sets, generalizing fractal sets, and help us to link the theory of IFS with non-stationary subdivision schemes.

\vfill\eject
\section{\bf Preliminaries} For a nonspecialist, we mention here the concepts, notation and basic results concerning traditional
 IFS and provide a brief outline of subdivision. For a detailed exposition the reader may consult \cite{B1,H} and \cite{CDM,DL} respectively.

\subsection{Basics of iterated function systems}\hfill

\medskip
Let $(X,d)$ be a complete metric space. For a function $f: X \to X$, we define
the Lipschitz constant associated with $f$ by
$$\text{Lip}(f) = \sup_{x,y \in X, x \neq y} \frac{d\big(f(x),f(y)\big)}{d(x,y)}.$$
A function $f$ is said to be Lipschitz function if $\text{Lip}(f) < + \infty$ and a contraction if $\text{Lip}(f) < 1$.
Let  $\mathbb{H}(X)$  be the collection of all nonvoid compact subsets of $X$. Then $\mathbb{H}$ is a metric space when endowed with the Hausdorff metric
$$ h (B,C) = \max \big\{d(B,C), d(C,B)\big\},$$ where $d(B,C)= \sup_{b \in B} d(b,C)= \sup_{ b \in B} \inf_{c \in C} d(b,c)$.
It is well-known that the metric space $\big(\mathbb{H}(X),h\big)$ is complete \cite{B2}.
\begin{definition}\label{defIFS}

An iterated function system, IFS for short, consists of a metric space $(X,d)$ and a finite family of continuous maps $f_i: X \to X$, $i \in \{1,2,\dots,n\}$. We denote such an IFS by $\mathcal{F}=\{X; f_i: i=1,2,\dots, n\}$.
\end{definition}
With the IFS $\mathcal{F}$ as above, one can associate a set-valued map referred to as Barnsley-Hutchinson operator. With a slight abuse of notation, we use the same symbol $\mathcal{F}$ for the IFS, the set of functions in the IFS, and for the Barnsley-Hutchinson operator defined below. Consider the function $\mathcal{F}: \mathbb{H}(X)\to \mathbb{H}(X)$
\begin{equation}\label{FX}
 \mathcal{F}(B) := \cup_{f \in \mathcal{F}} f(B),\ \ B\in \mathbb{H}(X),
\end{equation}
where $f(B):= \big\{f(b): b \in B \big\}$.
%It is well known that for a contractive IFS,  $\mathcal{F}$ in (\ref{FX}) is a contraction with contraction factor $L_{\mathcal{F}}$ \cite{B2}.
%If $f_i$ are contraction maps, the IFS is called contractive.
The contraction constant of $\mathcal{F}$ is \cite{B2}:
\begin{equation}\label{CLF}
L_{\mathcal{F}}=\max_{i=1,2,\dots, n} \text{Lip}(f_i).
\end{equation}
If $f_i$ are contraction maps, the IFS is contractive.
Therefore, by the Banach contraction principle we have
\begin{theorem}\label{theoremIFS}
Let $(X,d)$ be a complete metric space and $\mathcal{F}=\{X; f_i: i=1,2,\dots, n\}$ be an IFS with contraction constant $L_{\mathcal{F}}<1$. Then there exists a unique
set $A_{\mathcal{F}}$, such that $\mathcal{F}(A_{\mathcal{F}}) = A_{\mathcal{F}}$. Furthermore, for every $B_0 \in \mathbb{H}(X)$ the sequence $B_{k+1} =  \mathcal{F} (B_k)$ converges to $A_{\mathcal{F}}$ in $\mathbb{H}$. Also \cite{B2},
$$h(B_0,A_{\mathcal{F}}) = \frac{1}{1- L_{\mathcal{F}}}~ h(B_0, B_1).$$
\end{theorem}
\begin{remark}\label{remarkIFS}
\hfill
\begin{enumerate}
\item The set $A _{\mathcal{F}}$ appearing in the previous theorem is called the attractor of the IFS. The construction of $A _{\mathcal{F}}$ through iterations of the map $\mathcal{F}$ suggests the name iterated function system for $\mathcal{F}=\{X; f_i: i=1,2,\dots, n\}$.
\item The result of Theorem \ref{theoremIFS} holds even if $\mathcal{F}$ is not a contraction map, but an $\ell$-term composition of $\mathcal{F}$,
namely, $\mathcal{F}\circ \mathcal{F}\circ...\circ \mathcal{F}$ is a contraction map.
The $\ell$-term composition is a contraction if all the compositions of the form
\begin{equation}\label{lterm}
f_{i_1}\circ f_{i_2}\circ \cdot\cdot\cdot f_{i_\ell},\ \ \ i_j\in\{1,2,...,n\},
\end{equation}
are contractions.
\end{enumerate}
\end{remark}

\subsection{Basics of subdivision schemes}\label{BOS}\hfill

\medskip
A subdivision scheme is defined by a collection of real maps called refinement rules relative to a set of meshes of isolated points $$N_0 \subseteq N_1 \subseteq \dots \subseteq \mathbb{R}^s.$$
Each refinement rule maps real vector values defined on $N_k$ to real vector values defined on a refined net $N_{k+1}$.
Here we consider only scalar binary subdivision schemes, with $N_k=2^{-k}\mathbb{Z}^s$.
Given  a set of control points $p^{0}=\{p_j^0\in \mathbb{R}^m,\ \ j \in \mathbb{Z}^s\}$ at level $0$, a stationary binary subdivision scheme recursively defines new sets of points $p^k= \{p_j^k: j \in \mathbb{Z}^s\}$ at level $k \ge 1$, by the refinement rule
\begin{equation}\label{sa}
p_i^{k+1} = \sum_{j \in \mathbb{Z}^s} a_{i-2j} p_j^k,\ \ k\ge 0,
\end{equation}
or in short form,
$$p^{k+1}=S_a p^k,\ \ k\ge 0.$$
The set of real coefficients $a= \{a_j: j \in \mathbb{Z}^s\}$ that determines the refinement rule is called the mask of the scheme.
We assume that the support of the mask, $\sigma(a)=
\{j \in \mathbb{Z}^s:a_j \neq 0\}$, is finite.
$S_a$ is a bi-infinite two-slanted matrix with the entries $(S_a)_{i,j}=a_{i-2j}$.\\
A non-stationary binary subdivision scheme is defined formally as
$$p^{k+1}=S_{a^{[k]}} p^k,\ \ k\ge 0,$$
where the refinement rule at refinement level $k$ is of the form
\begin{equation}\label{sak}
p_i^{k+1} = \sum_{j \in \mathbb{Z}^s} a_{i-2j}^{[k]} p_j^k,\ \ i\in \mathbb{Z}^s.
\end{equation}
In a non-stationary scheme, the mask $a^{[k]}:= \{a_j^{[k]}: j \in \mathbb{Z}^s\}$  depends on the refinement level.
%If the $k$-th refinement step consists of $m$ refinement rules, the subdivision is said to be of arity $m$. It is noticed that higher arity %schemes have higher smoothness and approximation order while their support is smaller compared to lower arity schemes.
In univariate schemes $s=1$, there are two different rules in (\ref{sak}), depending on the parity of $i$.

In this paper we refer to two definitions of convergent subdivision. The first is the classical one in subdivision theory \cite{DL}:
\begin{definition}\label{C0conv}{\bf $C^0$-convergent subdivision}\\
A subdivision scheme is termed $C^0$-convergent if for any initial data  $p^0$ there exists a continuous function $f:\mathbb{R}^s \to \mathbb{R}^m$, such that
\begin{equation}\label{pktof}
\lim_{k\to\infty}\sup_{i\in \mathbb{Z}^s}|p_i^{k}-f(2^{-k}i)|=0,
\end{equation}
and for some initial data $f\ne 0$.
\end{definition}

\begin{remark}\label{remarkC0}
\hfill
\begin{enumerate}
\item The limit curve of a $C^0$-convergent subdivision is denoted by $p^\infty=S_a^\infty p^0$, and the function $f$ in Definition \ref{C0conv} specifies a parametrization of the limit curve.
%\item The convergence in (\ref{pktof}) is uniform in $i\in \mathbb{Z}^s$.
\end{enumerate}
\end{remark}
The analysis of subdivision schemes aims at studying the smoothness properties of the limit function $f$. For further reading see \cite{DL}.

We introduce here a weaker type of convergence using a set distance approach, influenced by IFS convergence:
\begin{definition}\label{hconv}{\bf $h$-convergent subdivision}\\
A subdivision scheme is termed $h$-convergent if for any initial data $p^0$ there
exists a set $p^\infty\subset \mathbb{R}^m$, such that
\begin{equation}\label{setlimit}
\lim_{k\to\infty}h(p^k,p^\infty)=0,
\end{equation}
where $h$ is the Euclidian-Hausdorff metric on $\mathcal{R}^m$.
The set $p^\infty$ is termed the $h$-limit of the subdivision scheme.
\end{definition}

It is clear that any $C^0$-convergent subdivision is also $h$-convergent.

\bigskip
In both subjects, IFS and subdivision, one is interested in the limits of iterative processes.
A connection between IFS and stationary subdivision is established in \cite{SLG}. In order to extend this connection to the case of non-stationary subdivision we investigate below the convergence properties of sequences of transformations in a metric space.

\section{\bf Sequences of transformations and Trajectories}
This section is intended to  introduce trajectories induced by a sequence of transformations and establish some elementary properties.
\par Let $(X,d)$ be a complete metric space. Consider a sequence of continuous transformations $\{T_i\}_{i \in N}$, $T_i: X \to X$.
\begin{definition}{\bf Forward and backward procedures:}

For the sequence of maps $\{T_i\}_{i \in N}$ we define forward and backward procedures
\begin{enumerate}
\item $ \Phi_k=T_k \circ T_{k-1} \circ \dots \circ T_1,$
\item $ \Psi_k=T_1 \circ T_2 \circ \dots \circ T_k.$
\end{enumerate}
\end{definition}

\begin{definition}{\bf Forward and backward trajectories:}

Induced by the forward and the backward procedures, we define consequent
forward and backward trajectories in $X$, starting from $x\in X$, $\{\Phi_k(x)\}$ and $\{\Psi_k(x)\}$,
%\begin{enumerate}
%\item $ \Phi_k(x)=T_k \circ T_{k-1} \circ \dots \circ T_1(x),\ \ k \in \mathbb{N},$
%\item $ \Psi_k(x)=T_1 \circ T_2 \circ \dots \circ T_k(x),\ \ k \in \mathbb{N}.$
%\end{enumerate}
\begin{equation}\label{PhiPsi}
\begin{aligned}
\Phi_k(x)=T_k \circ T_{k-1} \circ \dots \circ T_1(x)=T_k\circ\Phi_{k-1}(x),\ \ k \in \mathbb{N},\\
\Psi_k(x)=T_1 \circ T_2 \circ \dots \circ T_k(x)=\Psi_{k-1}\circ T_k(x),\ \ k \in \mathbb{N}.
\end{aligned}
\end{equation}
\end{definition}

In the present section we study the convergence of both types of trajectories. Later on we demonstrate the application
of both types to sequences of function systems and to subdivision.
%\begin{remark}
%In the rest of the paper, we shall consider trajectory as the one given in item (i) of the foregoing definition for reasons that will be evident %in due course.
%\end{remark}
%\begin{proposition}
%The trajectory of $\{T_i\}_{i\in \mathbb{N}}$  under a compact subset $A$  forms a sequence in $\big(\mathbb{H}(X),h\big)$.
%\end{proposition}
%\begin{proof}
%Proof follows at once from the well-known fact that the continuous image of a compact set is compact.
%\end{proof}
To state our next proposition, let us first introduce the following definition.
\begin{definition}
Two sequences $\{x_i\}_{i \in \mathbb{N}}$ and $\{y_i\}_{i \in \mathbb{N}}$ in a metric space $(X,d)$ are said to be asymptotically similar if $d(x_i,y_i) \to 0$ as $i \to \infty$. We denote this relation by
\begin{equation}
\{x_i\}\sim \{y_i\}.
\end{equation}
\end{definition}
\begin{proposition}\label{Equivalence}{\bf Asymptotic similarity of trajectories}\\
Let $\{T_i\}_{i \in \mathbb{N}}$ be a sequence of transformations on $X$, where each $T_i$ is a Lipschitz map with Lipschitz constant $s_i$. If $\lim_{ k \to \infty} \prod_{i=1}^k s_i =0$, then
%for both the forward and the backward trajectories, all trajectories are asymptotically similar. %I.e.,
for any $x,y\in X$,
\begin{equation}
\begin{aligned}
\{\Phi_k(x)\}\sim \{\Phi_k(y)\},\\
\{\Psi_k(x)\}\sim \{\Psi_k(y)\}.
\end{aligned}
\end{equation}
%the trajectories $\{\Phi_k(x)\}$ and $\{\Psi_k(x)\}$ $X$ form equivalent sequences in $\big(\mathbb{H}(X),h\big)$.
\end{proposition}

Note that the condition $\lim_{k\to\infty} \prod_{i=1}^k s_i =0$ does not imply $\limsup_{k \to \infty} s_k <1$.
\begin{proof}
The proof is similar for the forward and the backward trajectories.
Let $x,y\in X$ and consider the trajectories $\{\Psi_k(x)\}$ and $\{\Psi_k(y)\}$.
Using the fact that $T_i$ is a Lipschitz map with Lipschitz constant $s_i$, we get
\begin{equation}\label{prodsi}
\begin{aligned}
d\big(\Psi_k(x),\Psi_k(y)\big)\ \le\ & s_1 d(\big(T_2 \circ T_3 \circ \dots \circ T_k(x),T_2 \circ T_3 \circ \dots \circ T_k(y)\big)) \\
& \le s_1s_2d(\big(T_3 \circ T_4 \circ \dots \circ T_k(x),T_3 \circ T_4 \circ \dots \circ T_k(y)\big)) ... \\
& \le \big(\prod_{i=1}^k s_i\big) d(x,y),
\end{aligned}
\end{equation}
from which the result follows.
\end{proof}

\begin{remark}\label{remarkT}
The condition $\lim_{ k \to \infty} \prod_{i=1}^k s_i =0$ stated in Proposition \ref{Equivalence} does not guarantee
convergence of the trajectories $\{\Phi_k(x)\}$.
\end{remark}

If $T_i=T$ $\forall i\in \mathbb{N}$, and $T$ is a Lipschitz map with Lipschitz constant $\mu<1$, then both types of trajectories are just
the fixed-point iteration trajectories $\{T^k(x)\}$, where $T^k$ is the $k$-fold autocomposition of $T$ which converge to a unique limit for any starting point $x$. It is known from the Banach contraction principle that $\{T^k(x)\}$ converges to a unique limit irrespective of the starting point $x$.
The question now arises regarding the convergence of general trajectories, i.e., which conditions guarantee the convergence of the
forward and the backward trajectories. Having in mind the applications to fractal generation and to subdivision, we would like to know
which trajectories yield new types of fractals or new types of limit functions. Let us start with the forward trajectories
$\{\Phi_k(x)\}$.
%We introduce the following conditions on the transformations $\{T_i\}_{i \in \mathbb{N}}$:

\begin{definition}\label{def3}{\bf Invariant set of $\{T_i\}$.}\\
We call $C\subseteq X$ an invariant set of a sequence of transformations $\{T_i\}_{i \in \mathbb{N}}$ if
\begin{equation}\label{C}
\forall~ x\in C,\ \ T_i(x)\in C,\ \ \forall~ i\in \mathbb{N}.
\end{equation}
\end{definition}

\begin{lemma}\label{lemma3}
Consider a sequence of transformations $\{T_i\}_{i \in \mathbb{N}}$. If there exists $q$ in $X$ such that for every $x\in X$
\begin{equation}\label{C2}
d(T_i(x),q)\le \mu d(x,q)+M,\ \ 0\le\mu <1,\ \ M\in \mathbb{R}_+,
\end{equation}
then the ball of radius $\frac{M}{1-\mu}$ centered at $q$, $B\big(q,\frac{M}{1-\mu}\big)$, is an invariant set of
$\{T_i\}_{i \in \mathbb{N}}$.
\end{lemma}
\begin{proof}
For $x\in B\big(q,\frac{M}{1-\mu}\big)$
\begin{equation}\label{C3}
d(T_i(x),q)\le \mu d(x,q)+M \le \mu \frac{M}{1-\mu}+M = \frac{M}{1-\mu}.
\end{equation}
\end{proof}

\begin{remark}\label{remarkC}
Under the conditions of Lemma \ref{lemma3}, any ball $B(q,R)$ with $R>\frac{M}{1-\mu}$ is also an invariant set of
$\{T_i\}_{i \in \mathbb{N}}$. This follows since $M$ in (\ref{C2}) can be replaced by any $M^*>M$.
\end{remark}

\begin{example}\label{Ex1}
Consider a sequence of affine transformations on $\R^m$ of the form
\begin{equation}\label{Ex1eq}
T_i(x)=A_ix+b_i,\ \ i\in \mathbb{N},
\end{equation}
where $\{A_i\}$ are $m\times m$ matrices with $\|A_i\|_2\le \mu<1$, and $\|b_i\|_2\le M$.
Then the conditions of Lemma \ref{lemma3} are satisfied with $q=0$, and thus $C=B\big(0, \frac{M}{1-\mu}\big)$ is
an invariant set of $\{T_i\}_{i \in \mathbb{N}}$.
\end{example}

%Assumptions (1) and (3) above are also motivated by our applications.
%The interesting questions concern the convergence of both types of trajectories, with relation to the fixed point of $T$.

\begin{proposition}\label{forwardconvergence}{\bf Convergence of forward trajectories}\\
Let $\{T_i\}_{i \in \mathbb{N}}$ be a sequence of transformations on $X$, with a compact invariant set $C$,
and assume $\{T_i\}_{i \in \mathbb{N}}$ converges uniformly on $C$ to a Lipschitz map $T$ with Lipschitz constant $\mu<1$.
Then for any $x\in C$ the trajectory $\{\Phi_i(x)\}_{i \in \mathbb{N}}$
converges to the fixed-point $p$ of $T$, namely,
\begin{equation}
\lim_{k\to\infty}d(\Phi_k(x),p)=0.
\end{equation}
\end{proposition}
\begin{proof}
Denoting $\epsilon_i=\sup_{x\in C}d(T_i(x),T(x))$, $i\in \mathbb{N}$, it follows that
\begin{equation}\label{epsto0}
\lim_{i\to\infty}\epsilon_i=0.
\end{equation}
Since $T$ is a Lipschitz map with Lipschitz constant $\mu<1$, the fixed-point iterations $\{T^k(x)\}$ converge to a
unique fixed-point $p\in X$ for any starting point $x$. It also follows that $C$ is an invariant set of $T$.
Starting with $x\in C$, we have that
$\{\Phi_k(x)\}\subseteq C$. Using the triangle inequality in $\{X,d\}$ and the Lipschitz property of $T$, we have
\begin{equation}
\begin{aligned}
d(\Phi_{k+m}(x),T^m\Phi_k(x))=
d(T_{k+m}\circ T_{k+m-1}\circ ...\circ T_{k+1}\circ \Phi_k(x),T^m\Phi_k(x))\le \\
d(T_{k+m}\circ T_{k+m-1}\circ ...\circ T_{k+1}\circ \Phi_k(x),T\circ T_{k+m-1}\circ ...\circ T_{k+1}\circ \Phi_k(x))+ \\
d(T\circ T_{k+m-1}\circ ...\circ T_{k+1}\circ \Phi_k(x),T^2\circ T_{k+m-2}\circ ...\circ T_{k+1}\circ \Phi_k(x))+ \\
... \\
+d(T^{m-1}\circ T_{k+1}\circ \Phi_k(x),T^m\Phi_k(x))\le \\
\epsilon_{k+m} + \mu \epsilon_{k+m-1}+ \mu^2\epsilon_{k+m-2}+...+ \mu^{m-1}\epsilon_{k+1}\le \\ %\sum_{i=1}^m \epsilon_{k+i}.
\max_{1\le i\le m}\{\epsilon_{k+i}\}\times {\frac {1} {1-\mu}}.
\end{aligned}
\end{equation}
Now we use the relation
\begin{equation}
d(\Phi_{k+m}(x),p)\le d(\Phi_{k+m}(x),T^m\Phi_k(x))+d(T^m\Phi_k(x),p).
\end{equation}
The result follows by observing that for $k$ large enough $\max_{1\le i\le m}\{\epsilon_{k+i}\}$ can be made as small as needed
(by (\ref{epsto0})), and for that $k$, for a large enough $m$, $d(T^m\Phi_k(x),p)$ is as small as needed.
\end{proof}

In Section \ref{IFS} we consider trajectories of transformations $\{T_i\}$ defined by function systems, and we look for the
attractors of such trajectories. We refer to such systems as non-stationary function systems, and we apply them to generate new fractals.
Proposition \ref{forwardconvergence} implies that in the case of forward trajectories, if $T_i\to T$ as $i\to \infty$, the limit of the forward trajectories is the attractor of the IFS corresponding to the limit function system, and hence not new. Let us now examine the
backward trajectories $\{\Psi_k(x)\}$, and establish conditions for their convergence.

\begin{proposition}\label{BTp2}{\bf Convergence of backward trajectories}
Let $\{T_i\}_{i \in \mathbb{N}}$ be a sequence of transformations on $X$, with a compact invariant set $C$,
and assume each $T_i$ is a Lipschitz map with Lipschitz
constant $s_i$. If $\sum_{k=1}^\infty \prod_{i=1}^k s_i <\infty$, then the backward trajectories
$\{\Psi_k(x)\}$, with
$ \Psi_k=T_1 \circ T_2 \circ \dots \circ T_k,\ \ k \in \mathbb{N},$ converge for any starting point $x\in C$ to a unique limit in $C$.
\end{proposition}
\begin{proof}

By (\ref{PhiPsi}) and the relation in (\ref{prodsi})
\begin{equation*}
\begin{split}
d\big(\Psi_{k+1}(x),\Psi_k(x) \big) = &~ d \big( \Psi_k (T_{k+1}(x)), \Psi_k(x) \big) \\
\le &~ \big(\prod_{i=1}^k s_i \big)d\big(T_{k+1}(x),x \big).\\
\end{split}
\end{equation*}
For $m,k \in \mathbb{N}$, $m>k$, we obtain
\begin{equation}\label{e1}
\begin{split}
d\big(\Psi_m(x),\Psi_k(x)\big) \le &~ d\big(\Psi_m(x),\Psi_{m-1}(x)\big) + \dots + d\big(\Psi_{k+2}(x),\Psi_{k+1}(x) \big)+d\big(\Psi_{k+1}(x),\Psi_k(x) \big)\\
\le &~ \big(\prod_{i=1}^{m-1}s_i\big) d\big(T_m(x),x\big)+\dots + \big(\prod_{i=1}^{k+1}s_i\big) d\big(T_{k+2}(x),x \big)+ \big(\prod_{i=1}^{k}s_i\big) d\big(T_{k+1}(x), x \big).
\end{split}
\end{equation}
For $i\in\mathbb{N}$, $T_i(x)\in C$ $\forall x\in C$, which implies that $d(T_i(x),x)\le M$ $\forall x\in C$, where $M$ is the diameter of $C$. Since $\sum_{k=1}^\infty \prod_{i=1}^k s_i < \infty$,  Eq. (\ref{e1}) asserts that
$d\big(\Psi_m(x),\Psi_k(x)\big) \to 0$ as $ k \to \infty$. That is, $\{\Psi_k(x) \}_{k \in \mathbb{N}}\subseteq C$ is a Cauchy sequence, and
due to the completeness of $\{X,d\}$, it is convergent $\forall x\in C$.
The uniqueness of the limit is derived by the equivalence of all trajectories as proved in Proposition \ref{Equivalence}.
\end{proof}

\begin{remark}\label{remark2}
In view of (\ref{PhiPsi},
the result of Proposition \ref{BTp2} holds under the milder assumption that $C$ is an invariant set of $\{T_i\}_{i\ge I}$, for some $I\in\mathbb{N}$.
\end{remark}

\begin{remark}\label{remark3}
{\bf Differences between forward and backward trajectories}
\begin{enumerate}
\item Note that if $T_i\to T$ and $T$ has Lipschitz constant $\mu<1$, then
$$\sum_{k=1}^\infty \prod_{i=1}^k s_i < \infty ,$$
and both the forward and the backward trajectories converge.
\item The condition $\lim_{k\to\infty}\prod_{i=1}^k s_i=0$ is sufficient for the asymptotic similarity result of both forward and backward trajectories. Under the stronger condition
$\sum_{k=1}^\infty \prod_{i=1}^k s_i < \infty$ and the existence of a compact invariant set, we get convergence for the backward trajectories.
\item In many cases, the backward trajectories converge, while the forward trajectories do not converge. To demonstrate this
let the metric space be $\mathbb{R}$ with $d(x,y)=|x-y|$, and
let us consider the simple sequence of contractive
transformations $T_{2i-1}(x)=x/2$, $T_{2i}=x/2+c$, $i\ge 1$. The backward trajectories converge to the fixed point of $S_1=T_1\circ T_2$, which is
$2c/3$. The forward trajectories have two accumulation points, which are the fixed point of $S_1$, i.e., $2c/3$, and the fixed point of
$S_2=T_2\circ T_1$, which is $4c/3$.
\end{enumerate}
\end{remark}

%*****************************************************************************************************************************

\section{\bf Trajectories of Sequences of Function Systems}\label{IFS}

Generalizing the classical IFS we consider a sequence of function systems, SFS in short,
and its trajectories.

\medskip
Let $(X,d)$ be a complete metric space. Consider an SFS $\{\mathcal{F}_i\}_{i\in \mathbb{N}}$ defined by
$$ \mathcal{F}_i = \big\{X; f_{1,i}, f_{2,i}, \dots, f_{n_i,i} \big \},$$
where $f_{r,i}: X \to X$ are continuous maps. The associated set-valued maps are given by
$$\mathcal{F}_i: \mathbb{H}(X) \to \mathbb{H}(X); \quad \mathcal{F}_i(A)= \cup_{r=1}^{n_i} f_{r,i}(A).$$
Denoting $s_{r,i}=\text{Lip}(f_{r,i})$, for $r=1,2,\dots,n_i$,
we recall that as in (\ref{CLF}), the contraction factors of $\mathcal{F}_i$ in $(\mathbb{H}(X),h)$
is $L_{\mathcal{F}_i}=\max_{r=1,2,\dots, n_i} s_{r,i}\equiv s_i$.
The traditional IFS theory deals with the attractor, namely, the set which is the `fixed-point' of a map $\mathcal{F}$. In this section we consider the trajectories of the SFS maps $\{\mathcal{F}_i\}_{i\in \mathbb{N}}$, which we refer to as forward and backward SFS trajectories
\begin{equation}
\Phi_k(A)=\mathcal{F}_k \circ \mathcal{F}_{k-1} \circ \dots \circ \mathcal{F}_1(A),\ \ \
\Psi_k(A)=\mathcal{F}_1 \circ \mathcal{F}_{2} \circ \dots \circ \mathcal{F}_k(A),\ \ k \in \mathbb{N},
\end{equation}
respectively.

As presented in Section \ref{sect1}, $\mathbb{H}(X)$, endowed with the Hausdorff metric $h$, is a complete metric space
if $(X,d)$ is complete.

The first observation is a corollary of Proposition \ref{Equivalence}:

\begin{corollary}\label{IFSequiv}{\bf Asymptotic similarity of SFS trajectories}\\
Consider an SFS defined by $\mathcal{F}_i=\big\{X; f_{1,i}, f_{2,i}, \dots, f_{n_i,i} \big \}$, $i \in \mathbb{N}$,
where $f_{r,i}: X \to X$ are Lipschitz maps. Further assume that the corresponding contraction factors $\{L_{\mathcal{F}_i}\}$ for the set-valued maps $\{\mathcal{F}_i\}$ on $(\mathbb{H}(X),h)$ satisfy $\lim_{ k \to \infty} \prod_{i=1}^k L_{\mathcal{F}_i} =0$.
Then all the forward trajectories of $\{\mathcal{F}_i\}$ are asymptotically similar, and all the backward trajectories of $\{\mathcal{F}_i\}$ are asymptotically similar.
%Then for both the forward and the backward trajectories of $\{\mathcal{F}_i\}$ in $\{\mathbb{H}(X),h\}$, all %trajectories are asymptotically similar.
\end{corollary}

The next result is a corollary of Proposition \ref{forwardconvergence}:
\begin{corollary}{\bf Convergence of forward SFS trajectories}\label{SFSforward}\\
Let $\{\mathcal{F}_i\}_{i\in \mathbb{N}}$ be as in Corollary \ref{IFSequiv},
with equal number of maps, $n_i=n$, and let $\mathcal{F}=\{X; f_r: r=1,2,\dots, n\}$.
Assume  that there exists $ C\subseteq X$, a compact invariant set of $\{f_{r,i}\}$ and
that for each $r=1,2, \dots,n$, the sequence $\{f_{r,i}\}_{i\in \mathbb{N}}$ converges
uniformly to $f_r$ on $C$ as $ i \to \infty$.
Also assume that $\mathcal{F}$
has a contraction factor $L_{\mathcal{F}}<1$ .
Then the forward trajectories
$\{\Phi_k(A)\}$
converge for any initial set $A\subseteq C$ to the unique attractor of $\mathcal{F}$ .
\end{corollary}

\iffalse
\begin{proof}
Define  $\tilde{\mathcal{F}}: \mathbb{H}(X) \to \mathbb{H}(X)$ by
$$\tilde{\mathcal{F}}(A) = \cup_{r=1}^n \tilde{f_r} (A).$$
We have
\begin{equation*}
\begin{split}
h\big(\mathcal{F}_i(A), \tilde{\mathcal{F}}(A)\big)=&~ h\Big(\cup_{r=1}^n  f_{r,i} (A), \cup_{r=1}^n \tilde{f_r} (A)\Big)\\
\le &~ \max_{r=1,2,\dots,n} h \Big(  f_{r,i} (A),  \tilde{f_r} (A) \Big).
\end{split}
\end{equation*}
Since $\{f_{r,i}\}_{i\in \mathbb{N}}$ converges uniformly to $\tilde{f_r}$ on $X$, Lemma \ref{nonstatIFSl1} in conjunction with the previous inequality implies  that  $\{\mathcal{F}_i\}_{i \in \mathbb{N}}$ converges  $\tilde{\mathcal{F}}$ uniformly on $\mathbb{H}(X)$. Now  the conclusion follows exactly as in Proposition \ref{nonstatIFSp2}.
\end{proof}
\fi

\begin{remark}\label{remark4}
The forward trajectories of the SFS in Corollary \ref{SFSforward} converge to the fractal set (attractor) associated with $\mathcal{F}$ (see \cite{B1}). This observation implies that forward
trajectories of a converging SFS do not produce any new entities.
%\end{enumerate}
\end{remark}

Backward trajectories of SFS do not seem natural. However, as they converge under mild conditions, even if the
SFS $\{\mathcal{F}_i\}_{i \in \mathbb{N}}$ does not converge to a contractive function system, their limits, or attractors, may constitute
new entities, different from the known fractals which are self similar.

\begin{corollary}\label{backSFS}{\bf Convergence of backward SFS trajectories}

Let $\{\mathcal{F}_i\}_{i\in \mathbb{N}}$ and $\{L_{\mathcal{F}_i}\}$ be as in Corollary \ref{IFSequiv}.
Assume there exists $ C\subseteq X$, a compact invariant set of $\{f_{r,i}\}$, $r=1,...,n_i$, $i\in \mathbb{N}$,
and assume that $\sum_{k=1}^\infty \prod_{i=1}^k L_{\mathcal{F}_i}<\infty$.
Then the backward trajectories
$\{\Psi_k(A)\}$
converge, for any initial set $A\subseteq C$, to a unique set (attractor) $P\subseteq C$.
%Then the backward trajectories of $\{\mathcal{F}_i\}_{i\in \mathbb{N}}$ converge in $\{\mathbb{H}(X),h\}$ to a unique attractor.
\end{corollary}

\section{\bf Hidden fractals}

The fractal defined as the attractor of a single
$\mathcal{F}=\{X; f_r: r=1,2,\dots, n\}$ has the property of self-similarity, i.e., its local shape is unchanged under certain contraction maps.
The entities defined as the attractors of backward trajectories are more flexible. With a proper choice of $\{\mathcal{F}_i\}_{i\in \mathbb{N}}$
one can design different local behaviour under different contraction maps.
%And why the backward trajectories and not the forward trajectories?
Such a design relies on the observation that in a set defined by a sequence of contraction maps
\begin{equation}
\mathcal{G}_k(B)=\mathcal{F}_{1}\circ \mathcal{F}_{2}\circ \mathcal{F}_{3}\circ \cdot \cdot \cdot \circ \mathcal{F}_k(B),
\end{equation}
the first maps
$\mathcal{F}_{1}$,$\mathcal{F}_{2}$,$\mathcal{F}_{3}$,... determine the global shape of the set, while the details of
the local shape is determined by the last maps
$\mathcal{F}_{k}$,$\mathcal{F}_{k-1}$,$\mathcal{F}_{k-2}$,....
To understand this note, e.g., that the set $\mathcal{F}_k(B)$ is
undergoing a sequence of $k-1$ contraction maps. Therefore, its shape is not noticeable at larger scales.
The arrangement of the set $\mathcal{G}_k(B)$ is finally fixed by the maps $\{f_{1,1},f_{1,2},...,f_{1,n}\}$ of $\mathcal{F}_{1}$.
In general, if we scale by the contraction
factor of $\Psi_k=\mathcal{F}_1\circ\mathcal{F}_2\circ ...\circ\mathcal{F}_k$, we shall see the behavior of the
attractor of the backward trajectories of $\{\mathcal{F}_i\}_{i>k}$.

\begin{example}
As an example we consider an alternating sequence of maps $\{\mathcal{F}_i\}_{i\in \mathbb{N}}$, where for $10(j-1)<i\le 10j-5$,
$\mathcal{F}_i$ is the function system generating cubic polynomial splines, and for $10j-5<i\le 10j$ it is the function system generating the Koch fractal.
Both function systems are contractive of course. The forward trajectories do not converge (see Remark \ref{remark3}(3)),
while any backward trajectory is rapidly converging.
In Figure \ref{cubicKoch} we see on the left image of the global behavior of the limit which is a cubic spline behavior, and on the right image
the local behavior near $x=0$, which is like the Koch fractal. In higher resolution we have smooth behavior again, and so on.
Note that the scaling factor between the two images in Figure \ref{cubicKoch} is approximately $(1/2)^5$ which is the
contraction factor of the  first five mappings in $\{\mathcal{F}_i\}_{i\in \mathbb{N}}$.

%\iffalse
\begin{figure}[h]
    \includegraphics[width=2.5in]{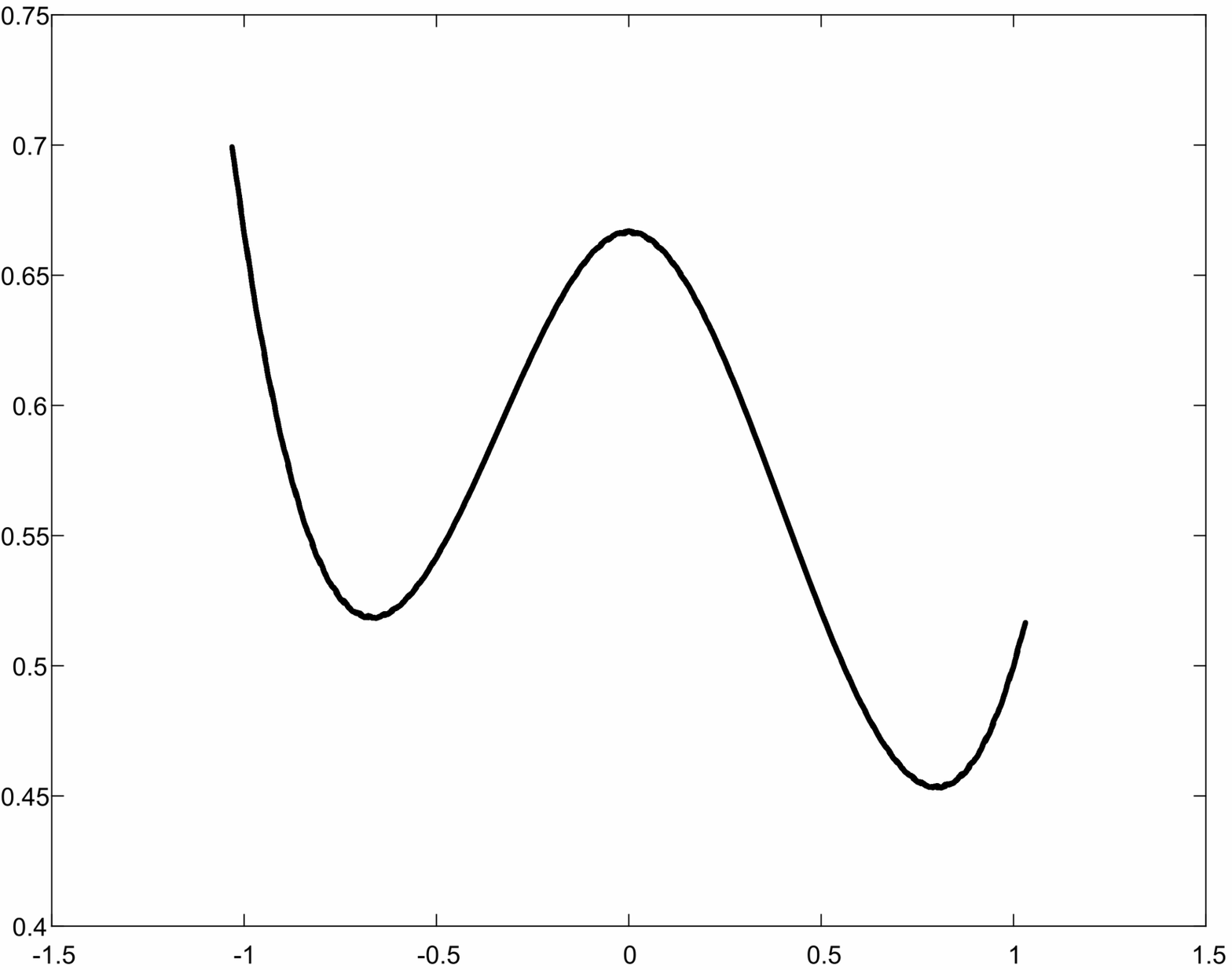}
    \hspace{10px}
    \includegraphics[width=2.5in]{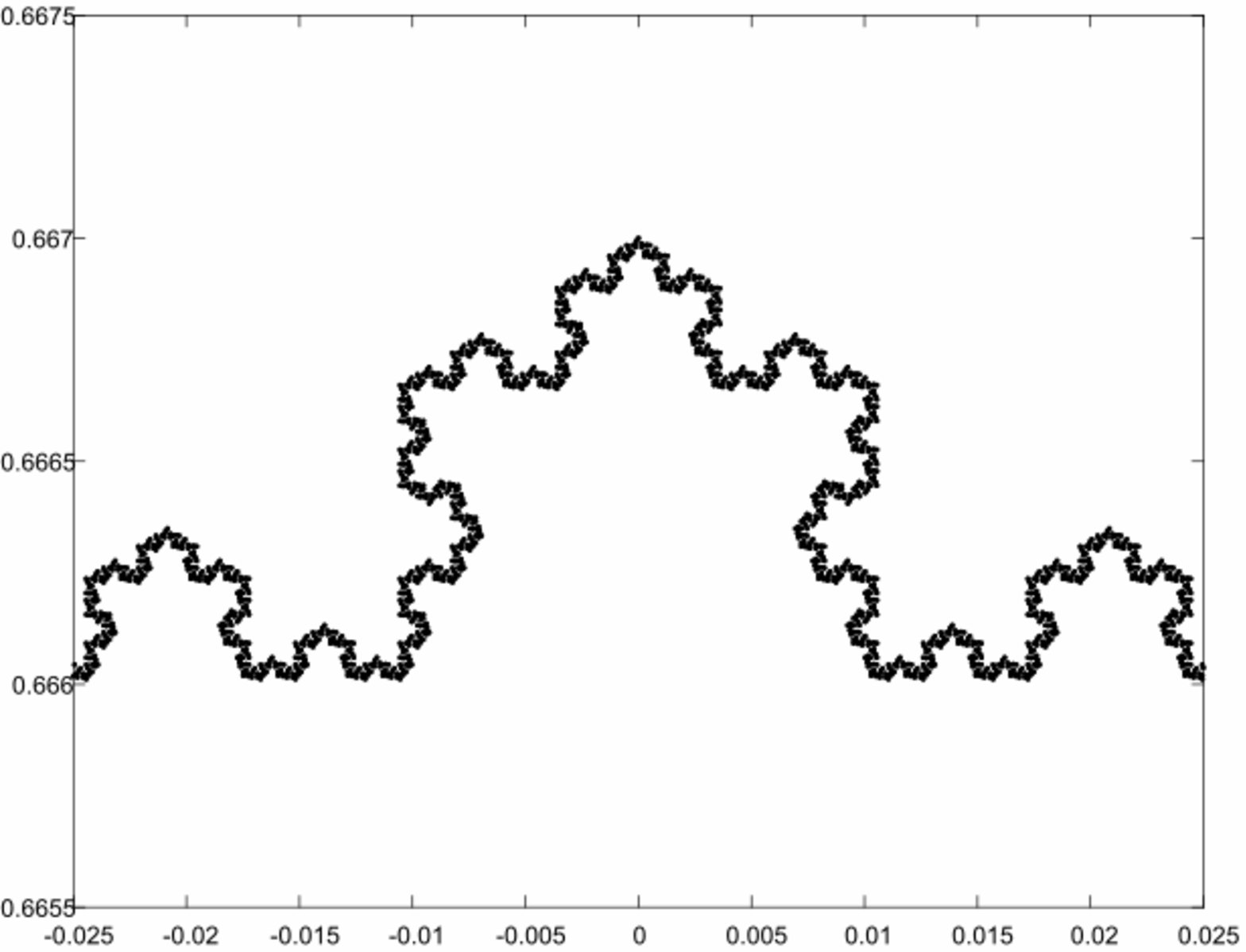}
    \caption{The cubic-Koch attractor: "Smooth" in one scale and "Fractal" in another.}
    \label{cubicKoch}
\end{figure}
%\fi
\end{example}

\section{\bf IFS related to convergent stationary subdivision}\label{6}\hfill

In this section we present IFS systems related to stationary subdivision schemes. The result in Subsections \ref{51}, \ref{52} are taken from \cite{SLG}. As in \cite{SLG} the discussion is
restricted to the case $s=1$, i.e., curves in $\mathbb{R}^m$.

\subsection{\bf $C^0$-convergent subdivision}\label{51}\hfill

\medskip

The connection between a $C^{0}$- convergent stationary subdivision for curves and IFS is presented in \cite{SLG}. In subdivision processes for curves ($s=1$)
one starts with an initial control polygon $p^0$, and the limit curve depends upon $p^0\subset \mathbb{R}^m$.
The attractor of the IFS does not depend upon the initial set. This dichotomy is resolved in \cite{SLG}
by defining an IFS related
to the subdivision operator $S$ which depends
upon $p^0$. The resulting IFS then converges to the relevant subdivision limit from any initial starting set. To understand the extension to non-stationary subdivision, let us first elaborate  the construction suggested in \cite{SLG} for the case of stationary subdivision for curves.

As presented in Section \ref{BOS}, a stationary binary subdivision scheme for curves in the plane ($s=1,\ m=2$) is defined by two refinement rules that take a set of control points
at level $k$, $p^k$, to a refined set at level $k+1$, $p^{k+1}$. For an infinite sequence $p^k$ this operation can be written in matrix form as
\begin{equation}\label{S}
p^{k+1}=Sp^k,
\end{equation}
where $S\equiv S_a$ is a two-slanted infinite martix with rows representing the two refinement rules, namely
$S_{i,j}=a_{i-2j}$, and $p^k$ is a matrix with $m$ columns and an infinite number of rows.
Given a finite set of control points, $\{p^0_j\in\mathbb{R}^m\}_{j=1}^n$ at level $0$, we are interested in computing the limit curve defined by these points. For a non-empty limit curve, $n$ should be larger than the support size $|\sigma(a)|$. We consider the sub-matrix of $S$ which operates on these points, and we cut from it two square $n\times n$ sub-matrices, $S_1$ and $S_2$, which
define all the $n_1$ resulting control points at level $1$. Note that  $S_1$ defines the transformation to the first $n$ points at level $1$, and
$S_2$ defines the transformation to the last $n$ points at level $1$. Of course there can be an overlap between these two vectors of points, namely $n_1< 2n$. Some examples of these sub-matrices are given in \cite{SLG}. We provide below the explicit forms of $S_1$ and $S_2$:

%\bf Explicit forms of $S_1$ an $S_2$}\\
We distinguish two types of masks, an even mask, with $2\ell$ elements, $a_{-\ell+1},...,a_\ell$, and an odd mask with $2\ell+1$ elements, $a_{-\ell},...,a_\ell$. For both cases we assume $n>\ell+1$. For both the even and the odd masks
\begin{equation}\label{S1}
S_1=\{a_{i-2j}\}_{i=\ell+1,\ j=1}^{\ \ \ \ell+n,\ \ \ n}.
\end{equation}
$S_2$ is different for odd and even masks. For an even mask
\begin{equation}\label{S2E}
S_2=\{a_{i-2j}\}_{i=n-\ell+2,\ j=1}^{\ \ \ 2n-\ell+1,\ \ n},
\end{equation}
and for an odd mask
\begin{equation}\label{S2O}
S_2=\{a_{i-2j}\}_{i=n-\ell+3,\ j=1}^{\ \ \ 2n-\ell+2,\ \ n}.
\end{equation}

Repeated applications of $S_1$ and $S_2$, define all the control points at all levels.
Therefore,
\begin{equation}\label{pinfty}
\bigcup_{i_1,i_2,...,i_k\in \{1,2\}}S_{i_k},...,S_{i_2}S_{i_1}p^0 \to p^\infty,\ \ \ as\ \ k \to \infty,
\end{equation}
where $p^\infty$ is the set of points on the curve defined by the subdivision process starting with $p^0$.

\begin{remark}\label{Union}{\bf Union of vectors of points}\\
$p^0$ is a vector of $n$ points in $\mathbb{R}^m$, and thus
each $S_{i_k},...,S_{i_2}S_{i_1}p^0$ is a vector of $n$ points in $\mathbb{R}^m$, which we regard as a set of
$n$ points in $\mathbb{R}^m$ . By
$\bigcup S_{i_k},...,S_{i_2}S_{i_1}p^0$ we mean the set in $\mathbb{R}^m$
which is the union of all these sets.
\end{remark}

\begin{remark}\label{PTP}{\bf Parameterizing the points in $p^\infty$}\\
To order the points of the set $p^\infty$ we introduce the following parametrization.
An infinite sequence $\eta=\{i_k\}_{k=1}^\infty$, $i_k\in\{1,2\}$ defines a vector of $n$ points in $\mathbb{R}^m$
\begin{equation}
\lim_{k\to \infty}S_{i_k},...,S_{i_2}S_{i_1}p^0=(q_1,...,q_n)^t,\ \ q_i\in \mathbb{R}^m.
\end{equation}
In case of a $C^{0}$-convergent subdivision, the differences between adjacent points tend to zero \cite{DGL}.
Therefore, all these $n$ points are the same point,
\begin{equation}\label{samepoint}
\lim_{k\to \infty}S_{i_k},...,S_{i_2}S_{i_1}p^0=(q_\eta,...,q_\eta)^t,\ \ q_\eta\in \mathbb{R}^m.
\end{equation}
We attach this point $q_\eta$ to the parameter value $x_\eta=\sum_{k=1}^\infty (i_k-1)2^{-k}\in [0,1]$.
\end{remark}

\subsection{\bf IFS related to stationary subdivision}\label{52}\hfill

\medskip
Here the metric space is $\{\mathbb{R}^n,d\}$ with $d(x,y)=\|x-y\|_2$, where $\|\cdot\|_2$ is the Euclidean norm.
The observation (\ref{pinfty}) leads in \cite{SLG} to the definition of an IFS with two maps on $X=\mathbb{R}^n$ (row vectors)
\begin{equation}\label{f1f2}
f_r(A)=AP^{-1}S_r P,\ \ \ r=1,2,
\end{equation}
where $P$ is an $n\times n$ matrix defined as follows:

\begin{enumerate}
\item The first $m$ columns of $P$ are the $n$ given control points $p^0$, which are points in $\mathbb{R}^m$.
\item The last column is a column of $1$'s.
\item The rest of the columns are defined so that $P$ is non-singular. We assume here that the control points $p^0$ do not all lie on an $m-1$ hyper plane so that the first $m$ columns of $P$ are linearly independent, and that the column of $1$'s is independent of the first $m$ columns.
\end{enumerate}

%As explained in \cite{SLG},
This special choice of $P$, together with the special definition of $f_1,f_2$ in (\ref{f1f2}), yields the following essential observations:
\begin{itemize}
\item
Since $S_1$ and $S_2$ have eigenvalue $1$, with right eigenvector $(1,1,...,1)^t$ which is also the last column of $P$, then
\begin{equation}\label{PSP}
P^{-1}S_rP
=\left(
\begin{array}{c|c}
G_r & 0 \\
\hline
v&1
\end{array}
\right)
, \ \ r=1,2,
\end{equation}
where $G_r$ are $(n-1)\times(n-1)$ matrices.
Denoting by $Q^{n-1}$ the $n-1$ dimensional hyperplane (flat) of
vectors of the form $(x_1,...,x_{n-1},1)$, it follows from (\ref{PSP}) that $f_r\ :\ Q^{n-1}\to Q^{n-1}$,
$r=1,2$.

%\item
%Since all the other eigenvalues of $S_1$ and $S_2$ are smaller than $1$, the maps $f_1,f_2$ are contractive on $Q^{n-1}$. Thus the IFS defined %by $\mathcal{F}=\{X; f_1,f_2\}$ has a unique attractor in $Q^{n-1}$.

\item
By applying the IFS iterations to the set $A=P$, using equation (\ref{FX}), we identify the candidate attractor as
\begin{equation}\label{Pinfty}
P^\infty=\lim_{k\to\infty}\bigcup_{i_1,i_2,...,i_k\in \{1,2\}}S_{i_k},...,S_{i_2}S_{i_1}P.
\end{equation}
Similarly to Remark \ref{Union}, the rows of $P^\infty$ constitute a set of points in $\mathbb{R}^n$.
By the structure of $P$, and in view of (\ref{pinfty}),  we observe that
$p^\infty$ is the set of points in $\mathbb{R}^m$ defined by the first $m$ components of the points (in $\mathbb{R}^n$) of $P^\infty$.
\end{itemize}

The above observations lead to the main result in \cite{SLG}, stated in the Theorem below. The original proof
in \cite{SLG} of this theorem has a flaw. We provide here a proof which serves us later in the discussion on non-stationary subdivision.

\begin{theorem}\label{ThSLG}
Let $S_a$ be a $C^{0}$-convergent subdivision, and let $p^0$ be a sequence of initial control points. Define the IFS $\mathcal{F}=\{X; f_1,f_2\}$
on $Q^{n-1}$, with $f_1,f_2$ defined in (\ref{f1f2}) and $S_1,S_2$ defined in (\ref{S1})-(\ref{S2O}). Then the IFS converges to a unique attractor in $Q^{n-1}$, and the first $m$ components of the points of this attractor constitute the limit curve $p^\infty=S_a^\infty p^0$.
\end{theorem}

\begin{proof}
Since all the eigenvalues of $S_1$ and $S_2$ which differ from $1$ are smaller than $1$, it follows that $\rho(G_r)<1$, $r=1,2$, where $\rho(G)$ is the spectral radius of $G$. This does not directly imply that the maps $f_1,f_2$ are contractive on $Q^{n-1}$. Following Remark \ref{remarkIFS}(2), to prove convergence of the IFS $\mathcal{F}$, we show that there exists an
$\ell$-term composition of $\mathcal{F}$ is a contraction map.
We notice that such an $\ell$-term composition of $\mathcal{F}$ is itself an IFS, with $2^\ell$ functions of the form
\begin{equation}\label{ftau}
f_\eta(A)=AP^{-1}S_{i_\ell}...S_{i_2}S_{i_1}P,\ \ \ \eta\in I_\ell,
\end{equation}
where $I_\ell=\{\eta=\{i_j\}_{j=1}^\ell$, $i_j\in\{1,2\}\}$.
$S_a$ is $C^{0}$-convergent, thus by Definition \ref{C0conv} it is also uniformly convergent. It follows from (\ref{samepoint})
that for any $\epsilon>0$, there exists 
$ \ell=\ell(\epsilon)$ such that for
any $\eta\in I_\ell$
\begin{equation}\label{stauP}
S_{i_\ell}...S_{i_2}S_{i_1}P=Q_\eta+E_\eta,
\end{equation}
where $Q_\eta$ is an $n\times n$ matrix of constant columns, and $\|E_\eta\|_\infty<\epsilon$. The last column of $Q_\eta$ is
$(1,1,...,1)^t$, and the last column of $E_\eta$ is the zero column.
Recalling that the last column of $P$ is the constant vector of $1$'s, and since $P^{-1}P=I_{n\times n}$, it follows that
\begin{equation}\label{Pm1stauP}
P^{-1}S_{i_\ell}...S_{i_2}S_{i_1}P=
\begin{pmatrix}
0&0&...&0&0\\
0&0&...&0&0\\
. & . & . & . & .\\
. & . & . & . & .\\
 0&0&...&0&0\\
 q_{\eta,1}&q_{\eta,2}&...&q_{\eta,n-1}&1\\
\end{pmatrix}
+P^{-1}E_\eta=
\left(
\begin{array}{c|c}
G_\eta & 0 \\
\hline
q_{\eta}&1
\end{array}
\right),
\end{equation}
Where $q_{\eta_j}(1,1,...,1)^t$ is the $j$-th column of $Q_\eta$.
It follows that $\|G_\eta\|_2\le \epsilon\|P^{-1}\|_2$.
Next we show that for $\epsilon$ small enough,
$f_\eta$ is contractive with respect to the Euclidean norm
in $Q^{n-1}$. Indeed, for $x,y\in \mathbb{R}^{n-1}$, $(x,1),(y,1)\in Q^{n-1}$, and
\begin{equation}\label{fetaxy}
d(f_\eta((x,1)),f_\eta((y,1)))=\|f_\eta((x,1)-(y,1))\|_2=\|f_\eta((x-y,0))\|_2=\|(x-y)^tG_\eta(x-y)\|_2.
\end{equation}
Choosing $\epsilon$ such that $\epsilon \|P^{-1}\|_2<1$, it follows that for all  $ \eta\in I_{\ell(\epsilon)}$, the map $f_\eta$ is contractive on $Q^{n-1}$, and the IFS defined by $\mathcal{F}$ is convergent.
\end{proof}

\begin{remark}
Theorem \ref{ThSLG} reveals the fractal nature of curves generated by subdivision. However, the self-similarity property of these curves is
not achieved in $\mathbb{R}^m$.
The self-similarity property is of $p^\infty$, as a set in $Q^{n-1}$. $p^\infty$ is the projection on $\mathbb{R}^m$
of this self similar entity in $Q^{n-1}$.
\end{remark}

\subsection{\bf A basis for convergent stationary subdivision}\hfill

\medskip

As presented above, and earlier in \cite{SLG}, the definition of an IFS for a $C^0$-convergent stationary subdivision involves the specific
given control points $p^0$. We observe that it is enough to consider one basic IFS, and its attractor can serve as a basis
for generating the limit of the subdivision process for any given $n$ control points $p^0$. Instead of the matrix $P$, we may define
any other non-singular $n\times n$ matrix with a last column of $1$'s. We choose the matrix
\begin{equation}\label{H}
H=
\begin{pmatrix}
1& 0& 0 & 0&...&1\\
0 &1& 0& 0&... &1 \\
0 &0& 1& 0&... &1 \\
 \cdot&\cdot&\cdot&\cdot&\cdot&1\\
 \cdot&\cdot&\cdot&\cdot&\cdot&1\\
 0 &0& 0&... &1&1 \\
0 &0& 0&... &0&1 \\
\end{pmatrix},
\end{equation}
and define the IFS with
\begin{equation}\label{f1f2H}
f_r(A)=AH^{-1}S_r H,\ \ \ r=1,2,
\end{equation}
As shown above, the attractor of this IFS is the union of $n\times n$ matrices
\begin{equation}\label{Hinfty}
\mathcal{H}^\infty=\lim_{k\to\infty}\bigcup_{i_1,i_2,...,i_k\in \{1,2\}}S_{i_k},...,S_{i_2}S_{i_1}H.
\end{equation}
In view of Remark \ref{Union}, $\mathcal{H}^\infty\subset Q^{n-1}$.

For any given control points $p^0$ we can simply calculate $p^\infty$ as the set
\begin{equation}\label{pinftyH}
p^\infty=\mathcal{H}^\infty H^{-1}p^0.
\end{equation}

\section{\bf SFS trajectories associated with non-stationary subdivision}\hfill

This research was motivated by the idea to adapt the framework of the previous section to non-stationary subdivision
processes. In binary non-stationary subdivision, as shown in (\ref{sak}), the
refinement rules may depend upon the refinement level, and can be written in matrix form as
\begin{equation}\label{Sk}
p^{k+1}=S^{[k]}p^k,
\end{equation}
where each $S^{[k]}\equiv S_{a^{[k]}}$ is a ``two-slanted" matrix.
As demonstrated in \cite{DL1}, non-stationary subdivision processes can generate interesting limits which cannot be generated by stationary schemes, e.g., exponential splines. Interpolatory non-stationary subdivision schemes can generate new types of orthogonal wavelets, as shown in \cite{DKLR}.

In the following we discuss the possible relation between non-stationary subdivision processes and SFS processes.
A necessary condition for the convergence (to a continuous limit) of a stationary subdivision scheme is the {\bf constants reproduction property},
namely,
\begin{equation}\label{Seeqe}
 Se=e, \ \ \ \
e=(...,1,1,1,1,1,...)^t .
\end{equation}
As explained in Section \ref{6}, this condition is used in \cite{SLG} in order to show that the maps defined
in (\ref{f1f2}) are contractive on $Q^{n-1}$.
This condition is not necessarily satisfied by converging
non-stationary subdivision schemes. It is also not a necessary condition for the construction of SFS related to non-stationary subdivision.

%Consider a non-stationary scheme $\{S_{a^{[k]}}\}$ with masks $\{a^{[k]}\}$ of a bounded support size, converging to a mask $a$,
%\begin{equation}\label{aktoa}
%\lim_{k\to \infty}a^{[k]}_j=a_j,\ \ \ j\in\sigma(a).
%\end{equation}
%In \cite{DL1} it is shown that the asymptotic speed of convergence in (\ref{aktoa}) is important for proving the convergence of the %non-stationary subdivision. For that, the notion of {\bf asymptotic equivalence} is introduced,
%\begin{equation}\label{aeaktoa}
%S_{a^{[k]}}\sim S_a\ \ \ if\ \ \ \sum_{k=1}^\infty|a^{[k]}_j-a_j|<\infty,\ \ \ j\in\sigma(a).
%\end{equation}
%It is shown that if $S_{a^{[k]}}\sim S_a
%$, and $S_a$ is convergent, then so is $S_{\{a^{[k]}\}}$.

\medskip
\subsection{\bf Constructing SFS mappings for non-stationary subdivision}\hfill

\medskip
In the following we assume that the supports of the masks $a^{[k]}$, $|\sigma(a^{[k]})|$, are of the same size, which is at most the number of initial control points.
As in the stationary case,
for a given set of control points, $\{p^0_j\}_{j=1}^n$, we define for each $k$ the two square $n\times n$ sub-matrices of each $S^{[k]}$, $S^{[k]}_1$ and $S^{[k]}_2$, in the same way as for a stationary scheme, by equations (\ref{S1}), (\ref{S2E}), (\ref{S2O}).
The points generated by the subdivision process are obtained by applying $S^{[1]}_1$ and $S^{[1]}_2$, to the initial control points vector $p^0$, and then applying
$S^{[2]}_1$ and $S^{[2]}_2$ to the two resulting vectors, and so on.
The set of points generated at level $k$ of the subdivision process is given by
\begin{equation}\label{nspk}
p^k=\bigcup_{i_1,i_2,...,i_k\in \{1,2\}}S^{[k]}_{i_k},...,S^{[2]}_{i_2}S^{[1]}_{i_1}p^0\ .
\end{equation}
If the subdivision is $C^0$-convergent or $h$-convergent, then
\begin{equation}\label{pktopinfty}
p^k\to p^\infty\ \ \ as\ \ k\to\infty,
\end{equation}
in the sense of Definitions \ref{C0conv}, \ref{hconv} respectively.
Here $p^\infty$ is the set of points defined by the non-stationary subdivision process starting with $p^0$.

Now we define the SFS $\{\mathcal{F}_k\}$, where
$\mathcal{F}_k=\big\{X; f_{1,k}, f_{2,k} \big \}$,
with the level dependent maps
\begin{equation}\label{f1f2k}
f_{r,k}(A)=AP^{-1}S^{[k]}_r P,\ \ \ r=1,2,
\end{equation}
where $P$ is the $n\times n$ matrix defined as in the stationary case.
%The last column in $P$ is chosen as a column of $1$'s.

\begin{remark}\label{QorR}
If the non-stationary scheme satisfies the constant reproduction property at every subdivision level,
then all the mappings in the SFS map $Q^{n-1}$ into itself (by (\ref{PSP})). If not, then the mappings are considered as maps on $\mathbb{R}^n$.
\end{remark}

Let us now follow a forward trajectory and a backward trajectory of $\Sigma\equiv\{\mathcal{F}_k\}$, starting from $A\subset \mathbb{R}^n$:
$$\mathcal{F}_k(A)=f_{1,k}(A)\cup f_{2,k}(A)=AP^{-1}S^{[k]}_1P\cup AP^{-1}S^{[k]}_2P,$$
and
$$\mathcal{F}_{j}(\mathcal{F}_k(A))=f_{1,j}(AP^{-1}S^{[k]}_1P\cup AP^{-1}S^{[k]}_2P)\cup f_{2,j}(AP^{-1}S^{[k]}_1P\cup AP^{-1}S^{[k]}_2P).$$
We note that
$$f_{r,j}(AP^{-1}S^{[k]}_iP)=AP^{-1}S^{[k]}_iPP^{-1}S^{[j]}_rP=AP^{-1}S^{[k]}_iS^{[j]}_rP.$$
Therefore,
$$\mathcal{F}_{j}(\mathcal{F}_k(A))=\bigcup_{r,i\in\{1,2\}}AP^{-1}S^{[k]}_iS^{[j]}_rP.$$
In the same way it follows that at the $k$th step of a forward trajectory of $\Sigma$ we generate the set
\begin{equation}\label{FASSP}
\mathcal{F}_{k}\circ\mathcal{F}_{k-1}\circ ...\circ\mathcal{F}_{2}\circ\mathcal{F}_1(A)=\bigcup_{i_1,i_2,...,i_k\in \{1,2\}}AP^{-1}S^{[1]}_{i_1},...,S^{[k-1]}_{i_{k-1}}S^{[k]}_{i_k}P.
\end{equation}
Similarly, the set generated at the $k$th step of a backward trajectory is
\begin{equation}\label{ASSP}
\mathcal{F}_{1}\circ\mathcal{F}_2\circ ...\circ\mathcal{F}_{k-1}\circ\mathcal{F}_k(A)=\bigcup_{i_1,i_2,...,i_k\in \{1,2\}}AP^{-1}S^{[k]}_{i_k},...,S^{[2]}_{i_2}S^{[1]}_{i_1}P.
\end{equation}
For the special backward trajectory with $A=P$ we obtain
\begin{equation}\label{PSSP}
\mathcal{F}_{1}\circ\mathcal{F}_2\circ ...\circ\mathcal{F}_{k-1}\circ\mathcal{F}_k(P)=\bigcup_{i_1,i_2,...,i_k\in \{1,2\}}S^{[k]}_{i_k},...,S^{[2]}_{i_2}S^{[1]}_{i_1}P.
\end{equation}

If the non-stationary subdivision scheme is either $C^0$-convergent or $h$-convergent, then,
in view of (\ref{nspk}), it follows that the first $m$ components in this special trajectory converge to the
limit $p^\infty$ of $\{S_{a^{[k]}}\}$, starting with $p^0$. The challenging question is finding for which classes of
non-stationary schemes {\bf all} the backward trajectories converge to the same limit. As we show later, and as explained in Remark \ref{remark4}, forward trajectories of $\Sigma$ are less interesting.

\subsection{\bf Attractors of forward and backward SFS trajectories for non-stationary subdivision}\hfill

\medskip

We consider forward and backward SFS trajectories for several cases of non-stationary subdivision schemes:

\begin{itemize}
  \item[Case (i)] A $C^{0}$-convergent non-stationary scheme $\{S_{a^{[k]}}\}$.
%(without assuming that the masks $\{a^{[k]}\}$ satisfy the constant reproduction property).
  \item[Case (ii)] A non-stationary scheme $\{S_{a^{[k]}}\}$ satisfying the constants reproduction property,
   with masks of the same support, converging to a mask $a$ of a $C^{0}$-convergent subdivision, i.e., $\sigma(a^{[k]})=\sigma(a)$, and
  \begin{equation}\label{aktoa}
\lim_{k\to \infty}a^{[k]}_j=a_j,\ \ \ j\in\sigma(a).
\end{equation}
  \item[Case (iii)]  A non-stationary scheme $\{S_{a^{[k]}}\}$ with masks $\{a^{[k]}\}$ satisfying the constants reproduction property, and corresponding $\{\mathcal{F}_k\}$ satisfying
  $\sum_{\ell=1}^\infty \prod_{k=1}^\ell L_{\mathcal{F}_k}<\infty$.
\end{itemize}
\medskip

%We recall here that as in (\ref{CLF}),
%\begin{equation}\label{CLFk}
%L_{\mathcal{F}_k}=\max_{i=1,2,\dots, n} \text{Lip}(f_{i,k}).
%\end{equation}

In Case (i) we do not assume that
the non-stationary subdivision scheme reproduces constants, nor do we assume that the masks $\{a^{[k]}\}$ converge to a limit mask.
Therefore, the associated SFS maps do not necessarily
map $Q^{n-1}$ to itself. We do assume that the non-stationary scheme is $C^0$-convergent.

\begin{theorem}\label{nsprop2}
Let $\{S_{a^{[k]}}\}$ be a non-stationary $C^0$-convergent subdivision scheme,
and let $\Sigma=\{\mathcal{F}_k\}_{k=1}^\infty$ be the SFS defined in (\ref{f1f2k}).
Then the backward trajectories of $\Sigma$ starting with $A\subset Q^{n-1}$ converge to a unique attractor. The first $m$ components of
the points of this attractor constitute the limit curve (in $\mathbb{R}^m$) of the non-stationary scheme defined
in (\ref{nspk})-(\ref{pktopinfty}).
\end{theorem}
\begin{proof}
Here we consider the SFS as mappings from $\mathbb{R}^n$ to itself.
Since $\{S_{a^{[k]}}\}$ converges, it immediately follows from (\ref{PSSP}) that the backward trajectory of $\Sigma$ initialized with $A=P$ converge. We would like to show that all the backward trajectories of $\Sigma$ initialized with an arbitrary set
of points $A\subset Q^{n-1}$ converge to the same limit.
We recall that the first $m$ columns of $P$ are the control points $p^0$.
Starting the backward trajectory of $\Sigma$ with $A=P$, it follows,
as discussed in Remark \ref{PTP}, that an infinite sequence $\eta=\{i_k\}_{k=1}^\infty$, $i_k\in\{1,2\}$, defines a vector of $n$ equal points in $\mathbb{R}^m$
\begin{equation}
q=\lim_{k\to \infty}S^{[k]}_{i_k},...,S^{[2]}_{i_2}S^{[1]}_{i_1}p^0=(q_\eta,...,q_\eta)^t,\ \ q_\eta=(q_{\eta,1},...,q_{\eta,m}),
\end{equation}
attached to a parameter value $x_\eta=\sum_{k=1}^\infty (i_k-1)2^{-k}$.
Starting the backward trajectory with a general set $A$ in $Q^{n-1}$, and following the same sequence $\sigma$, it follows from (\ref{ASSP}) that the limit is the $n\times m$ matrix $AP^{-1}q$. We recall that the last column of $P$ is a constant vector of $1$'s.
Since each column of $q$ is a constant vector of length $n$, and since $P^{-1}P=I_{n\times n}$, it follows that
\begin{equation}\label{pm1q}
P^{-1}q=
\begin{pmatrix}
0&0&...&0\\
0&0&...&0\\
 & . & & . \\
 & . & & . \\
 0&0&...&0\\
 q_{\eta,1}&q_{\eta,2}&...&q_{\eta,m}\\
\end{pmatrix}.
\end{equation}
For any row vector of the form $r=(r_1,r_2,...,r_{n-1},1)\in Q^{n-1}$, it follows from (\ref{pm1q}) that $rP^{-1}q=q_\eta$.
If $A$ represents a set of $N$ points in $Q^{n-1}$, i.e., the $n$th element in each row of $A$ is $1$, it follows that $AP^{-1}q$ represent $N$ copies of the same point $q_\eta$. That is, for any sequence of indices $\eta$, the limit of the corresponding trajectory is the same for any initial $A\subset Q^{n-1}$, and it is the limit point of the non-stationary subdivision attached to the parameter value $x_\eta$.
Comparing the trajectories displayed in (\ref{ASSP}) and (\ref{PSSP}), it follows that
\begin{equation}\label{limeqlim}
\lim_{k\to\infty}\mathcal{F}_{1}\circ\mathcal{F}_2\circ ...\circ\mathcal{F}_{k-1}\circ\mathcal{F}_k(A)=AP^{-1}
\lim_{k\to\infty}\mathcal{F}_{1}\circ\mathcal{F}_2\circ ...\circ\mathcal{F}_{k-1}\circ\mathcal{F}_k(P).
\end{equation}
Interchanging the order of $\lim_{k\to\infty}$ and $\bigcup_{i_1,i_2,...,i_k\in \{1,2\}}$
we conclude that both trajectories converge to the same limit for any $A\subset Q^{n-1}$.
\end{proof}

In Case (ii) we
consider a non-stationary scheme $\{S_{a^{[k]}}\}$ with masks converging to a mask $a$,
\begin{equation}\label{aktoa2}
\lim_{k\to \infty}a^{[k]}_j=a_j,\ \ \ j\in\sigma(a),
\end{equation}
with $S_a$ a convergent stationary scheme.
Thus
\begin{equation}\label{fikktofi}
\lim_{k\to \infty} f_{r,k}=f_r,\ \ \ r=1,2.
\end{equation}

Following Corollaries \ref{SFSforward} and \ref{backSFS}, we are now ready to discuss the convergence of forward and backward trajectories
of $\Sigma\equiv\{\mathcal{F}_k\}$.

\begin{corollary}\label{nscor1}{\bf Forward trajectories of $\{\mathcal{F}_k\}$:}
Let $\{S_{a^{[k]}}\}$ have the constant reproducing property, with masks $\{a^{[k]}\}$
of the same support size converging to the mask of a $C^{0}$-convergent
subdivision scheme $S_a$.
Then the forward trajectories of the SFS $\{\mathcal{F}_k\}$ defined above converge to the attractor $P^\infty$ of the IFS related to $S_a$.
\end{corollary}
\begin{proof}
Let $\mathcal{F}$ be the IFS related to $S_a$, and let $\{\mathcal{F}_k\}$ be the SFS related to the non-stationary scheme $\{S_{a^{[k]}}\}$.
Following the proof of Theorem \ref{ThSLG}, there exists an $\ell$ such that the $\ell$-term composition of $\mathcal{F}$, namely, $\mathcal{G}=\mathcal{F}\circ\mathcal{F}\circ ... \circ\mathcal{F}$, is a contraction map.
Let
\begin{equation}
\mathcal{G}_k=\mathcal{F}_{k\ell}\circ\mathcal{F}_{k\ell-1}\circ ... \circ\mathcal{F}_{(k-1)\ell+1},\ \ k\ge 1.
\end{equation}
Thus, $\mathcal{G}_k\to \mathcal{G}$ as $k\to\infty$, and $\exists K$ such that the maps
$\{\mathcal{G}_k\}_{k\ge K}$ are contractive.
In order to apply Corollary \ref{SFSforward} we need to show the existence of an invariant set $C$ for the maps $\{\mathcal{G}_k\}$.
Applying Example \ref{Ex1} we derive the existence of an invariant set $C_K$ for the maps $\{\mathcal{G}_k\}_{k\ge K}$. $C_K$ is a ball of radius $r$ in $Q^{n-1}$, centered at $q=(0,0,...,0,1)^t$.
By Remark \ref{remarkC}, any ball of radius $R>r$, centered at $q$, is also an invariant set of $\{\mathcal{G}_k\}_{k\ge K}$.

Using this observation in Corollary \ref{SFSforward}, implies that
all forward trajectories of $\{\mathcal{G}_k\}_{k\ge K}$ converge from any set in $Q^{n-1}$ to the attractor of $\mathcal{G}$.
%, which is $P^\infty$, the attractor of the IFS related to $S_a$.
In particular, for any set $A\in Q^{n-1}$, we can start the forward trajectory of $\{\mathcal{G}_k\}_{k\ge K}$ with the set
\begin{equation}
\mathcal{G}_{K-1}\circ\mathcal{G}_{K-2}\circ ...\circ\mathcal{G}_{2}\circ\mathcal{G}_1(A),
\end{equation}
and conclude that all forward trajectories of $\{\mathcal{G}_k\}_{k\ge 1}$ converge from any point in $Q^{n-1}$ to the attractor of the IFS related to $S_a$.
\end{proof}

\begin{remark}\hfill

\begin{enumerate}
\item
It is important to note that
in case the non-stationary scheme does not reproduce constants, the result in Corollary \ref{nscor1} does not necessarily hold.
To see this it is enough to consider the simple case where $S_i^{[k]}=S_i$, $i=1,2$, for $k\ge 2$, and only $S_1^{[1]}$ and
$S_2^{[1]}$ are different, and the corresponding $S_{a^{[1]}}$ does not reproduce constants. Then, in view of the expression (\ref{FASSP}), the forward trajectory
with $A=P$ converges to $S_1^{[1]}P^\infty\cup S_2^{[1]}P^\infty\ne P^\infty$, where $P^\infty$ is the attractor corresponding to the stationary subdivision with $S_1$ and $S_2$.
\item
The important conclusion from the above corollary is that forward trajectories of an SFS related to a non-stationary subdivision
with masks
converging to the mask of
a $C^0$-convergent subdivision do not produce any new attractors.
%The first $m$ components of its attractor constitute $p^\infty=S_a^\infty p^0$.
On the other hand, the backward trajectories related to such non-stationary subdivision schemes do generate new interesting
curves. See e.g. \cite{DL}.

\item%{\bf Backward trajectories revisited}
Under the conditions of Corollary \ref{nscor1}, it is proved in \cite{CDMM} that the non-stationary subdivision $\{S_{a^{[k]}}\}$ is $C^0$- convergent. Therefore, by Theorem \ref{nsprop2} the backward trajectories of $\Sigma$ starting with $A\subset Q^{n-1}$ converge to a unique attractor. This result follows from Corollary \ref{backSFS} as well.
\end{enumerate}
\end{remark}

%If the non-stationary subdivision scheme is either $C^0$-convergent or $h$-convergent, then,
%in view of (\ref{nspinfty}), it follows that the first $m$ components in this special trajectory converge to limit of $\{S_{a^{[k]}}\}$, starting %with $p^0$. We would like to prove that {\bf all} trajectories converge to the same limit. If we assume that the 'constant reproduction property' %is satisfied at all levels of $\{S_{a^{[k]}}\}$, then, using  Lemma \ref{nsifs}(3) and Corollary \ref{backSFS},
%the convergence of the backward trajectories of $\Sigma$ to a unique limit follows.

In case (iii), the mask of the subdivision schemes $\{S_{a^{[k]}}\}$ do not have
to converge to a mask of a $C^0$-convergent subdivision scheme. We still assume here that the non-stationary scheme reproduces constants, i.e., $(1,1,...,1)^t$ is an eigenvector of $S_1^{[k]}$ and $S_2^{[k]}$ with eigenvalue $1$, for $k\ge 1$. Let us denote by $\mu(S_{a^{[k]}})$ the maximal absolute value of the eigenvalues of $S_1^{[k]}$ and $S_2^{[k]}$ which differ from $1$.

\begin{corollary}\label{nscor5}
Consider a constant reproducing non-stationary scheme $\{S_{a^{[k]}}\}$ and let $\{\mathcal{F}_k\}_{k=1}^\infty$ be the SFS defined by (\ref{f1f2k}).
If $\ \sum_{\ell=1}^\infty \prod_{k=1}^\ell L_{\mathcal{F}_k}<\infty$
then:
\begin{enumerate}
\item All the backward trajectories of $\{\mathcal{F}_k\}$ converge to a unique attractor in $Q^{n-1}$.
\item The first $m$ components of this attractor constitute
the $h$-limit (in $\mathbb{R}^m$) of the scheme applied to the initial control polygon $p^0$.
\end{enumerate}
\end{corollary}

The proof follows directly from Corollary \ref{backSFS}.

%\iffalse
\subsection{Numerical Examples}\hfill

\medskip
\begin{example}(Case (i) and case (ii))
For our first example we consider a non-stationary subdivision which produces exponential splines.
It is convenient to view the mask coefficients $\{a_i\}$ of a subdivision scheme as the coefficients of a Laurent polynomial
$$a(z)=\sum_ia_iz^i.$$
The subdivision mask for generating cubic polynomial splines is
$$a(z)={(1+z)^4\over 8}={1\over 8}+{1\over 2}z+{3\over 4}z^2+{1\over 2}z^3+{1\over 8}z^4.$$
Following \cite{SLG}, the corresponding matrices $P$, $S_1$ and $S_2$, for $n=5$, are
\medskip
$$
P=
\begin{pmatrix}
x_1&y_1&1&0&1\\
x_2&y_2&0&1&1\\
x_3&y_3&0&0&1\\
x_4&y_4&0&0&1\\
x_5&y_5&0&0&1\\
\end{pmatrix},
\ \ \
S_1=
\begin{pmatrix}
{1\over 2}&{1\over 2}&0&0&0\\
{1\over 8}&{3\over 4}&{1\over 8}&0&0\\
0&{1\over 2}&{1\over 2}&0&0\\
0&{1\over 8}&{3\over 4}&{1\over 8}&0\\
0&0&{1\over 2}&{1\over 2}&0\\
\end{pmatrix},
\ \ \
S_2=
\begin{pmatrix}
0&{1\over 2}&{1\over 2}&0&0\\
0&{1\over 8}&{3\over 4}&{1\over 8}&0\\
0&0&{1\over 2}&{1\over 2}&0\\
0&0&{1\over 8}&{3\over 4}&{1\over 8}\\
0&0&0&{1\over 2}&{1\over 2}\\
\end{pmatrix}.
$$
\medskip
A related non-stationary subdivision is defined by the sequence of mask polynomials
\begin{equation}\label{ak}
a^{[k]}(z)={b_k(1+z)(1+c_kz)^3}, \ \ \ \text{with} \ \ c_k=\text{exp}({\lambda 2^{-k-1}}),\ b_k=1/(1+c_k)^3.
\end{equation}
The non-stationary subdivision $\{S_a^{[k]}\}$ generates exponential splines with integer knots,
piecewise spanned by $\{1,e^{\lambda x},xe^{\lambda x},x^2e^{\lambda x}\}.$
The matrices $S^{[k]}_1,S^{[k]}_2$ are
\medskip
$$
S^{[k]}_1=b_k
\begin{pmatrix}
3c_k^2+c_k^3&1+3c_k&0&0&0\\
c_k^3&3(c_k+c_k^2)&1&0&0\\
0&3c_k^2+c_k^3&1+3c_k&0&0\\
0&c_k^3&3(c_k+c_k^2)&1&0\\
0&0&3c_k^2+c_k^3&1+3c_k&0\\
\end{pmatrix},
$$
$$
S^{[k]}_2=b_k
\begin{pmatrix}
0&3c_k^2+c_k^3&1+3c_k&0&0\\
0&c_k^3&3(c_k+c_k^2)&1&0\\
0&0&3c_k^2+c_k^3&1+3c_k&0\\
0&0&c_k^3&3(c_k+c_k^2)&1\\
0&0&0&3c_k^2+c_k^3&1+3c_k\\
\end{pmatrix}.
$$
\medskip
We observe that $\lim_{k\to \infty} c_k=1$, and thus $\lim_{k\to\infty}a^{[k]}= a$.
The conditions for both Corollary \ref{nscor1} and Theorem \ref{nsprop2} are satisfied, and both forward and
backward trajectories of $\{\mathcal{F}_k\}$ converge.
The attractors of both forward and backward trajectories, for $\lambda=3$, are
presented in Figure \ref{nsfig}. The symmetric set is in Figure \ref{nsfig} is the attractor of the forward trajectory,
which is a segment of the cubic polynomial B-spline, and the non-symmetric set is the attractor of the backward trajectory,
and it is a part of the exponential B-spline.

\begin{figure}[!ht]
    \includegraphics[width=2.5in]{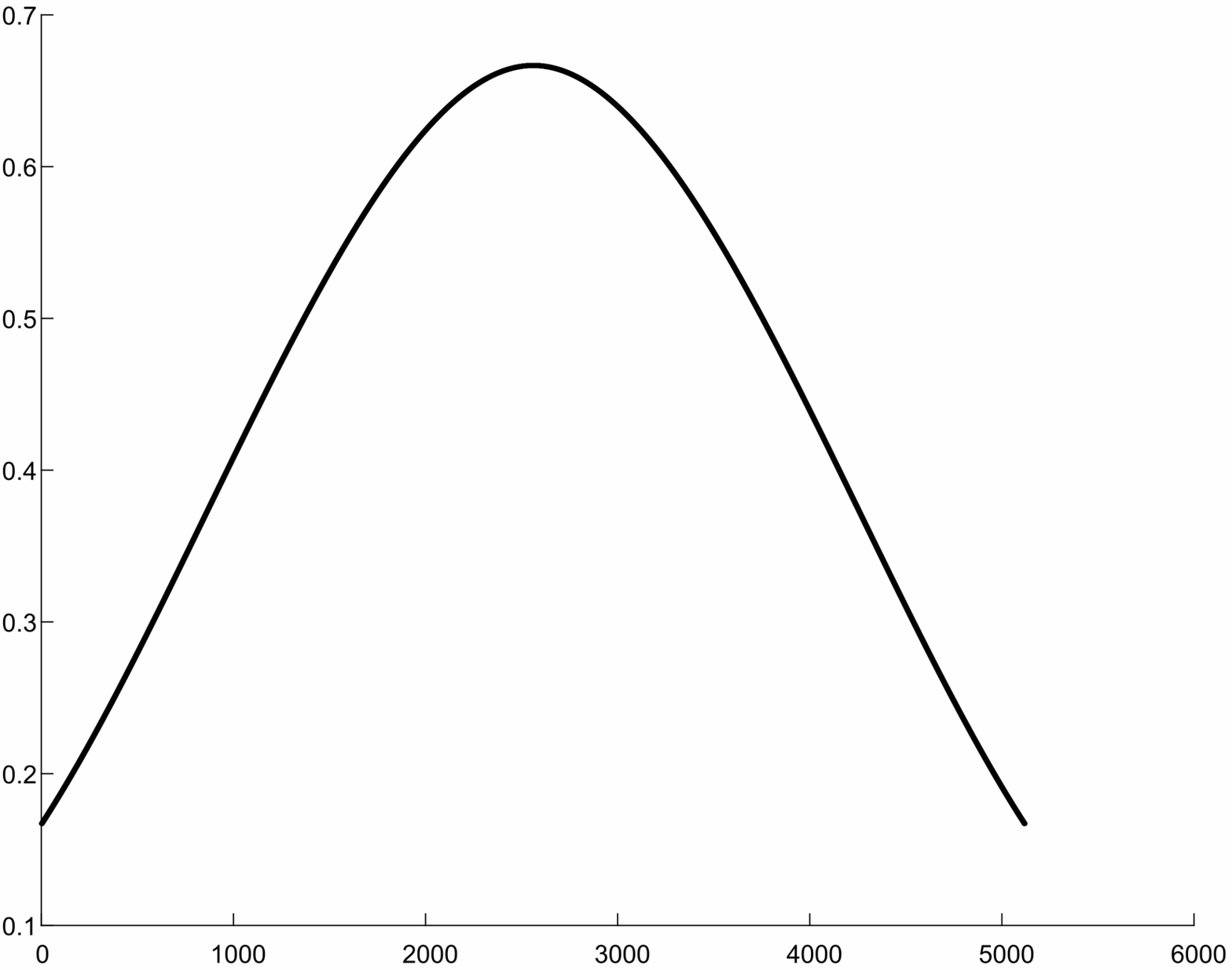}
    \hspace{10px}
    \includegraphics[width=2.5in]{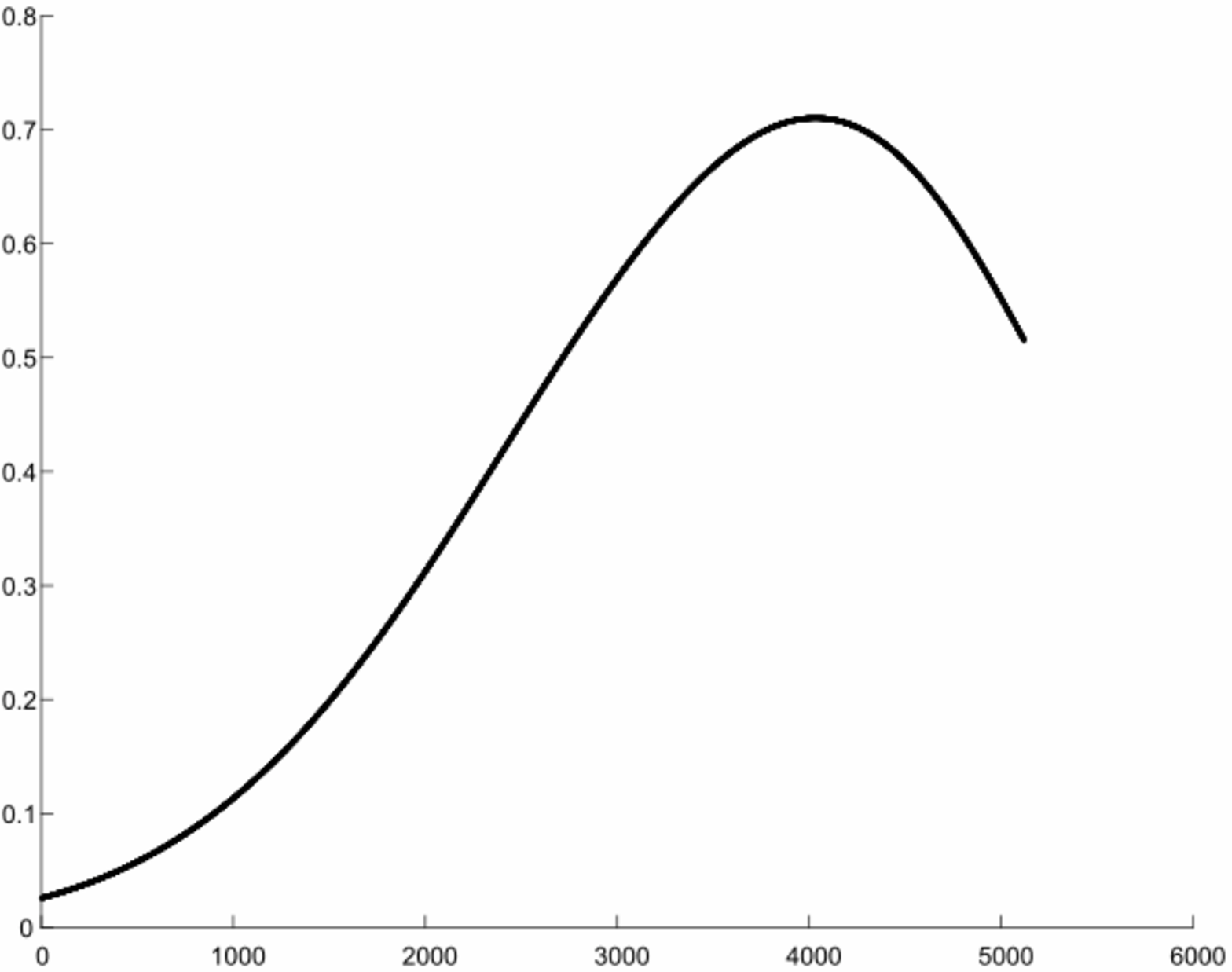}
    \caption{Left: Forward trajectory limit - cubic spline\\
    Right: Backward trajectory limit - exponential spline.}
    \label{nsfig}
\end{figure}
\end{example}.

\begin{example}(Case (iii)).
As we have learnt from Corollary \ref{backSFS}, backward SFS trajectories may converge under quite mild
conditions. In particular, an SFS derived from a non-stationary subdivision process, may converge even if it
is not asymptotically equivalent to a converging stationary process.
Let us consider the random non-stationary 4-point interpolatory subdivision process defined by the Laurent polynomials
\begin{equation}\label{randomns}
a^{[k]}(z)=-w_k(z^{-3}+z^3)+(0.5+w_k)(z^{-1}+z)+1,
\end{equation}
where $\{w_k\}_{k=1}^\infty$ are randomly chosen in an interval $I$. For the constant sequence $w_k=w$, this is
the Laurent polynomial representing the stationary $4$-point scheme presented in \cite{DGL}.
This random 4-point subdivision has been considered in \cite{Levin}, and it is shown there that the scheme is $C^1$ convergent
for $w_k\in [\epsilon, 1/8-\epsilon]$. Here we study the convergence for a larger interval $I$.
We define the SFS
$\mathcal{F}_k=\big\{\mathbb{R}^n; f_{1,k}, f_{2,k} \big \}$
where $f_{1,k}, f_{2,k}$ are define by (\ref{f1f2k}) with
the corresponding matrices $S^{[k]}_1,S^{[k]}_2$
$$
S^{[k]}_1=
\begin{pmatrix}
0&1&0&0&0&0\\
-w_k&0.5+w_k&0.5+w_k&-w_k&0&0\\
0&0&1&0&0&0\\
0&-w_k&0.5+w_k&0.5+w_k&-w_k&0\\
0&0&0&1&0&0\\
0&0&-w_k&0.5+w_k&0.5+w_k&-w_k\\
\end{pmatrix},
$$
$$
S^{[k]}_2=
\begin{pmatrix}
-w_k&0.5+w_k&0.5+w_k&-w_k&0&0\\
0&0&1&0&0&0\\
0&-w_k&0.5+w_k&0.5+w_k&-w_k&0\\
0&0&0&1&0&0\\
0&0&-w_k&0.5+w_k&0.5+w_k&-w_k\\
0&0&0&0&1&0\\
\end{pmatrix},
$$
and
$$
P=
\begin{pmatrix}
0&2&1&0&0&1\\
1&1&0&1&0&1\\
2&1&0&0&1&1\\
3&2&0&0&0&1\\
2&4&0&0&0&1\\
1&4&0&0&0&1\\
\end{pmatrix}.
$$
\end{example}

Considering Corollary \ref{backSFS} about the convergence of backward SFS trajectories,
we need the existence of a compact invariant set of $\{f_{r,i}\}$,
and that $\sum_{k=1}^\infty \prod_{i=1}^k L_{\mathcal{F}_i}<\infty$.
By numerical simulations we observe that for this example
$\sum_{k=1}^\infty \prod_{i=1}^k L_{\mathcal{F}_i}<\infty$ is satisfied if
$\{w_k\}$ are chosen according to a uniform random distribution in $I=[-b,b]$, with $0<b<0.86$.
We further conclude that for $\{w_k\}\in I$ there exists $m$ such that for any $i\in \mathbb{N}$,
$\prod_{i=k}^{k+m-1} L_{\mathcal{F}_i}<\mu<1$. Using Example \ref{Ex1} we can verify that there
exists a compact invariant set of the linear maps $\{\mathcal{A}_i\}$, where
$$\mathcal{A}_i=\mathcal{F}_i\circ \mathcal{F}_{i+1}\circ ....\circ \mathcal{F}_{i+m-1}.$$
By Corollary \ref{backSFS}, this guarantees the convergence of the
backward trajectories of $\{A_{km}\}$ to a unique attractor, and this implies the convergence
of the backward trajectories of $\{\mathcal{F}_i\}$.
Figures \ref{randw1}, \ref{randw2}, \ref{randw3} depict the convergence of the backward trajectories
$\{\Psi_k(A)\}$ of $\{\mathcal{F}_i\}$ for
$w_k\in [-0.2,0.2]$, $w_k\in [-0.4,0.4]$, $w_k\in [-0.8,0.8]$, respectively, and for $k=10,12,14$.
%\iffalse
\begin{figure}[h]
    \includegraphics[width=1.7in]{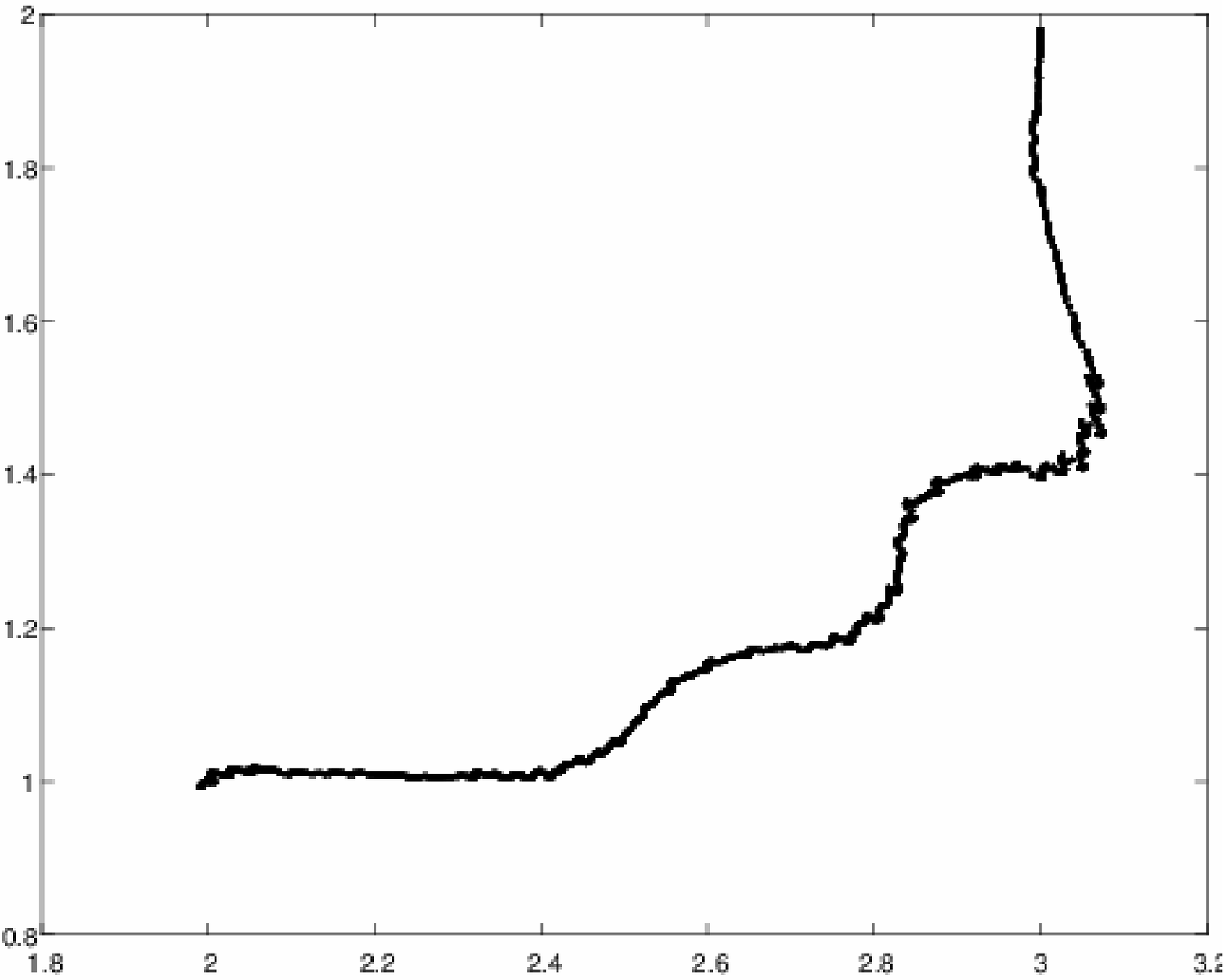}
    \hspace{10px}
    \includegraphics[width=1.7in]{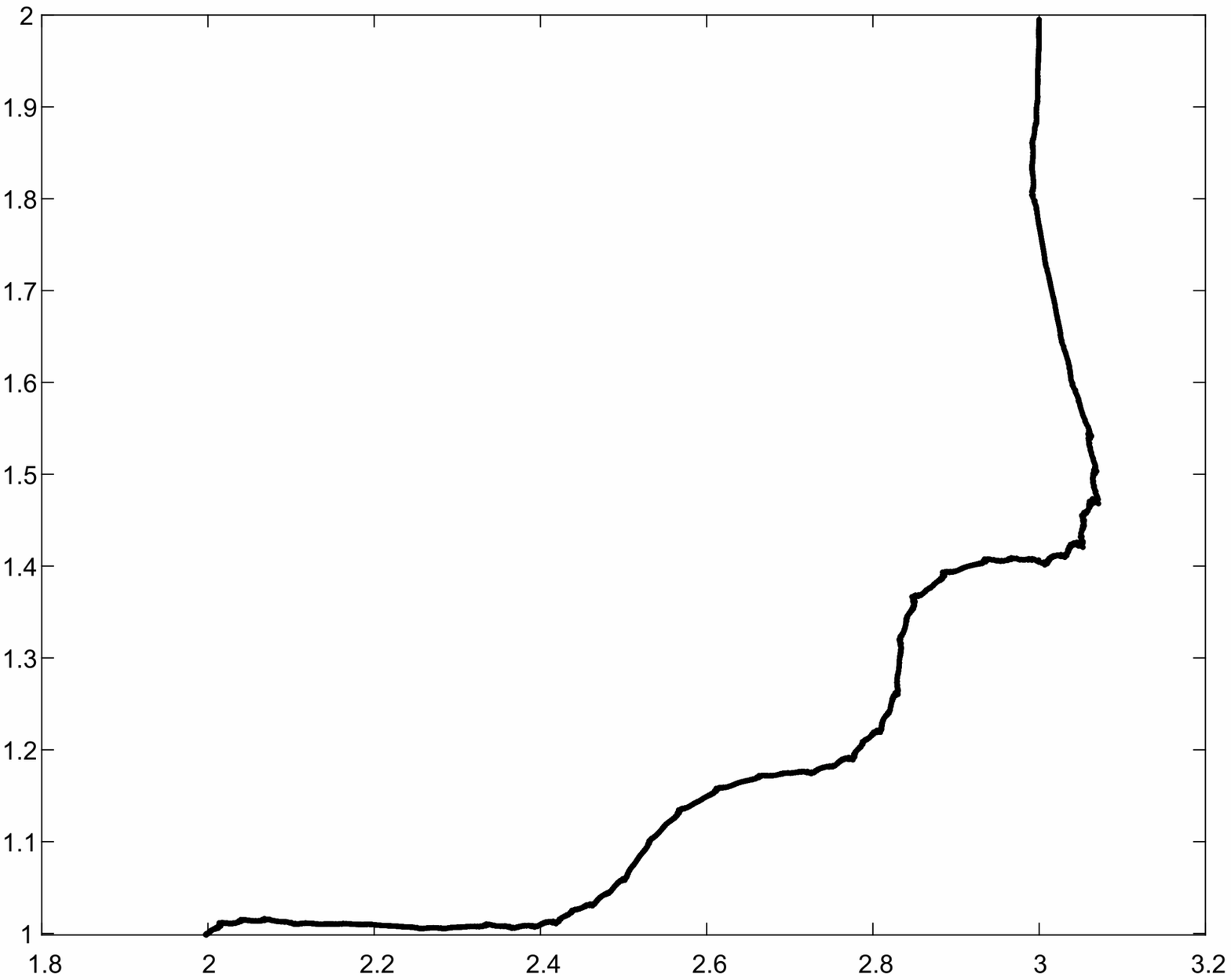}
    \hspace{10px}
    \includegraphics[width=1.7in]{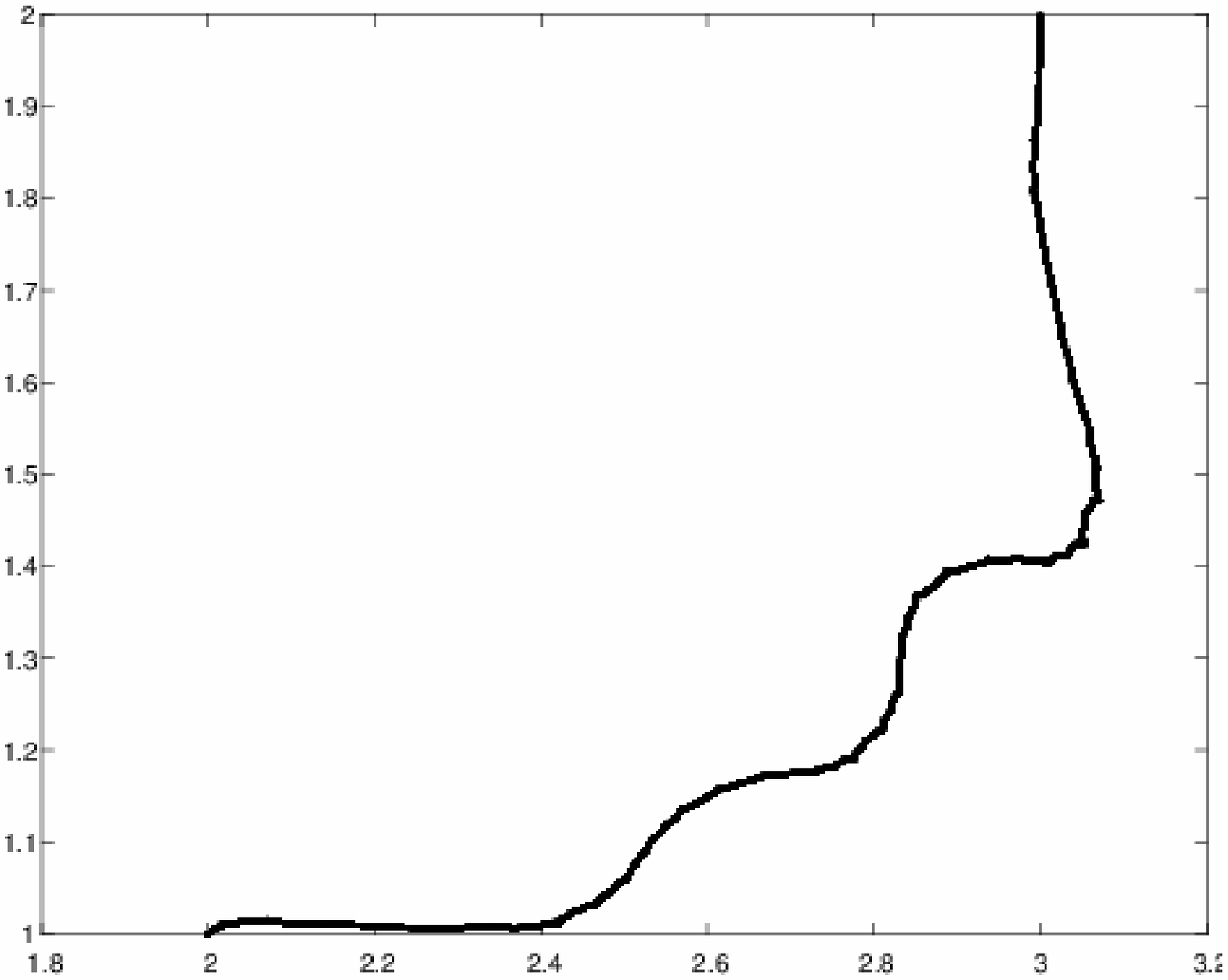}
    \caption{$w_k\in [-0.2,0.2]$; \ \ Backward trajectories:\ \ $\Psi_k(A)$, $k=10,12,14$.}
    \label{randw1}
\end{figure}
%\fi
\begin{figure}[h]
    \includegraphics[width=1.7in]{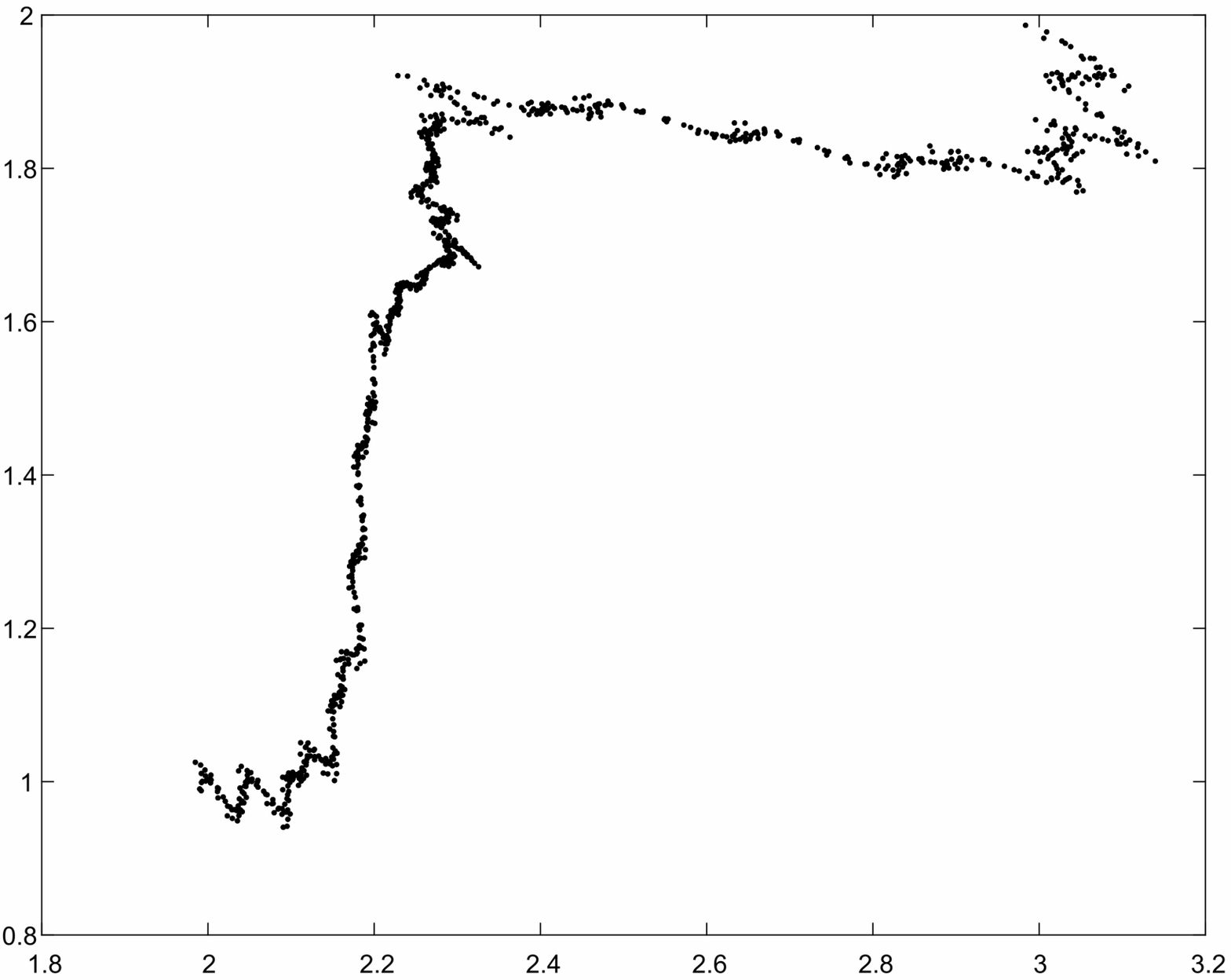}
    \hspace{10px}
    \includegraphics[width=1.7in]{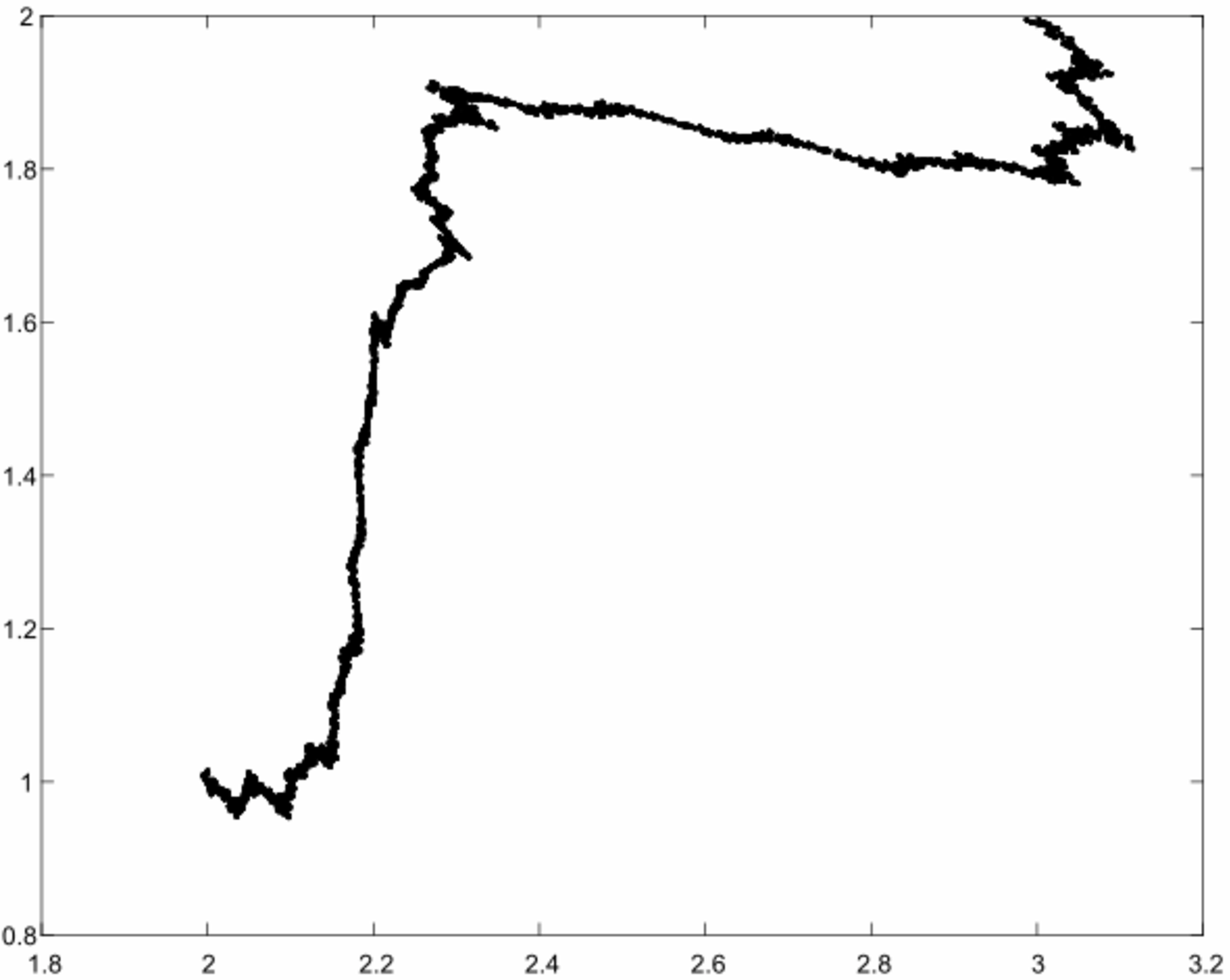}
    \hspace{10px}
    \includegraphics[width=1.7in]{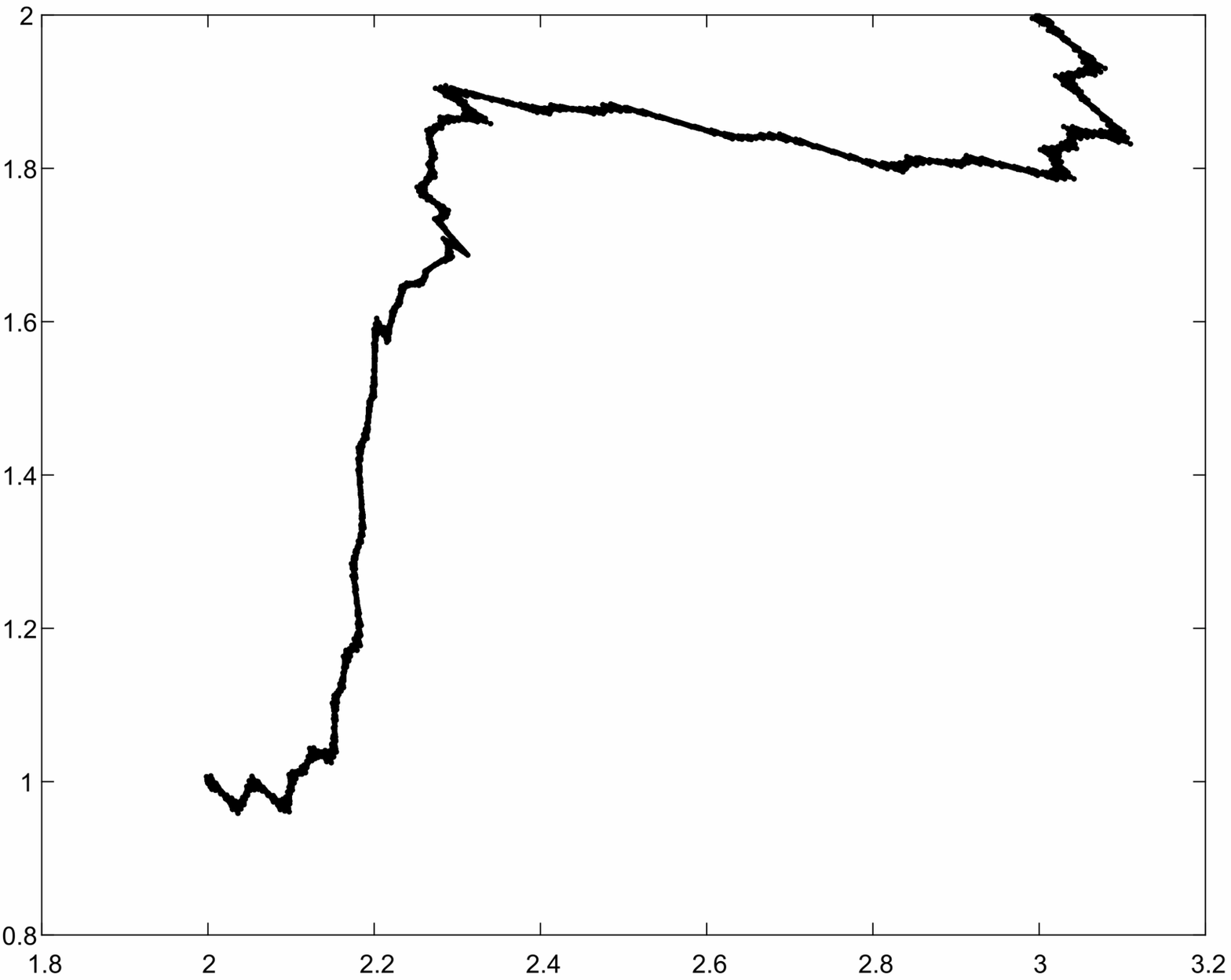}
    \caption{$w_k\in [-0.4,0.4]$; \ \ Backward trajectories:\ \ $\Psi_k(A)$, $k=10,12,14$.}
    \label{randw2}
\end{figure}

\begin{figure}[h]
    \includegraphics[width=1.7in]{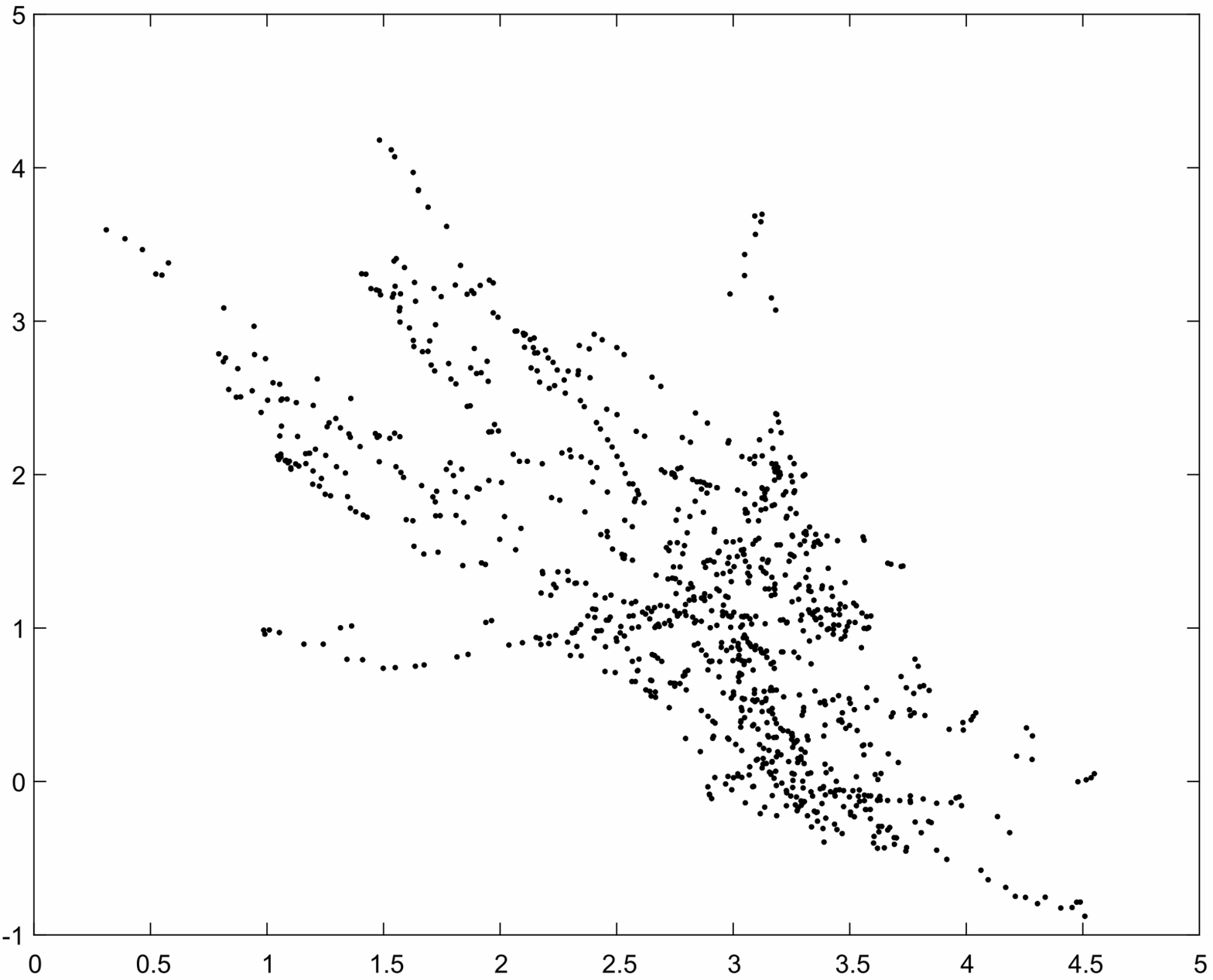}
    \hspace{10px}
    \includegraphics[width=1.7in]{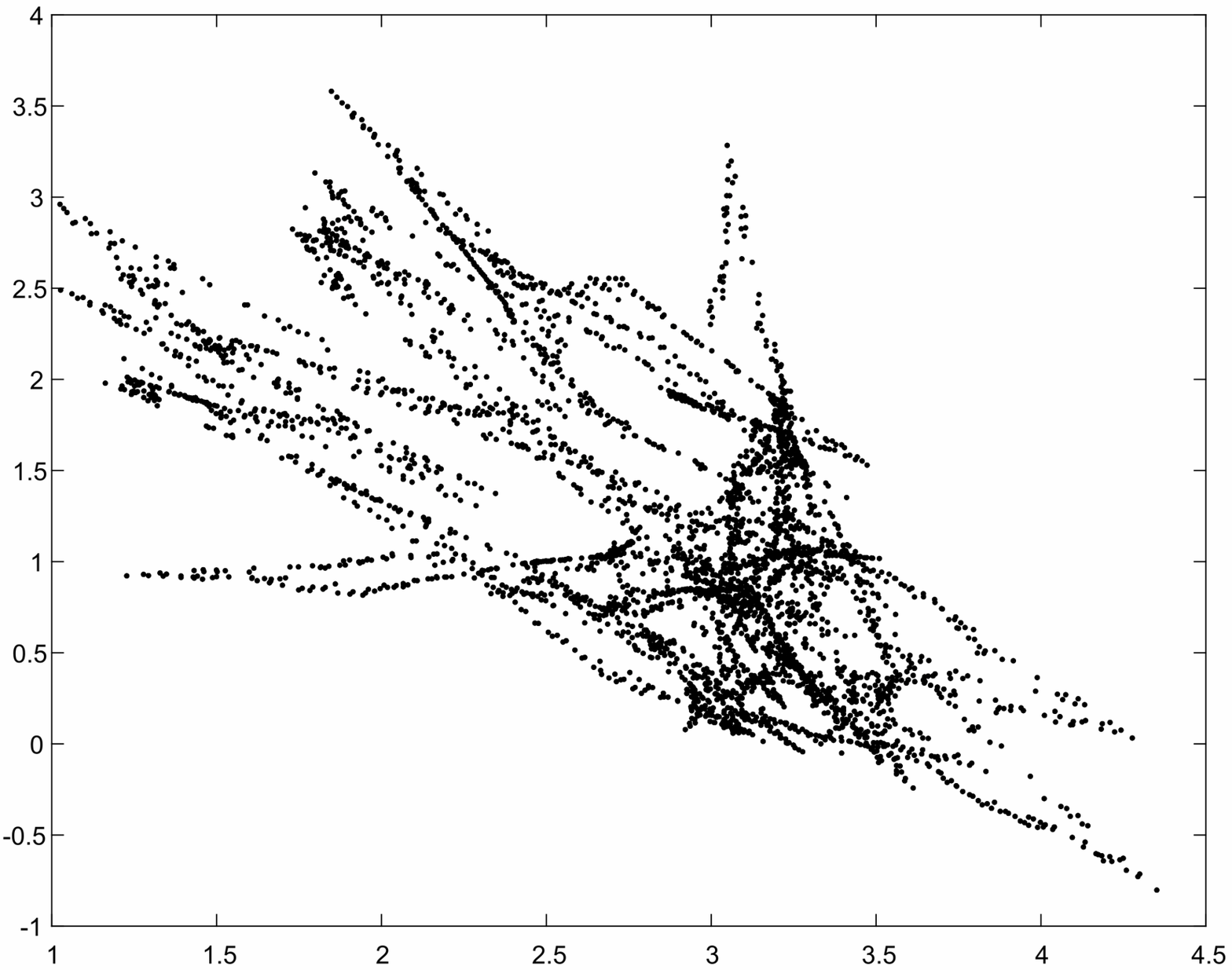}
    \hspace{10px}
    \includegraphics[width=1.7in]{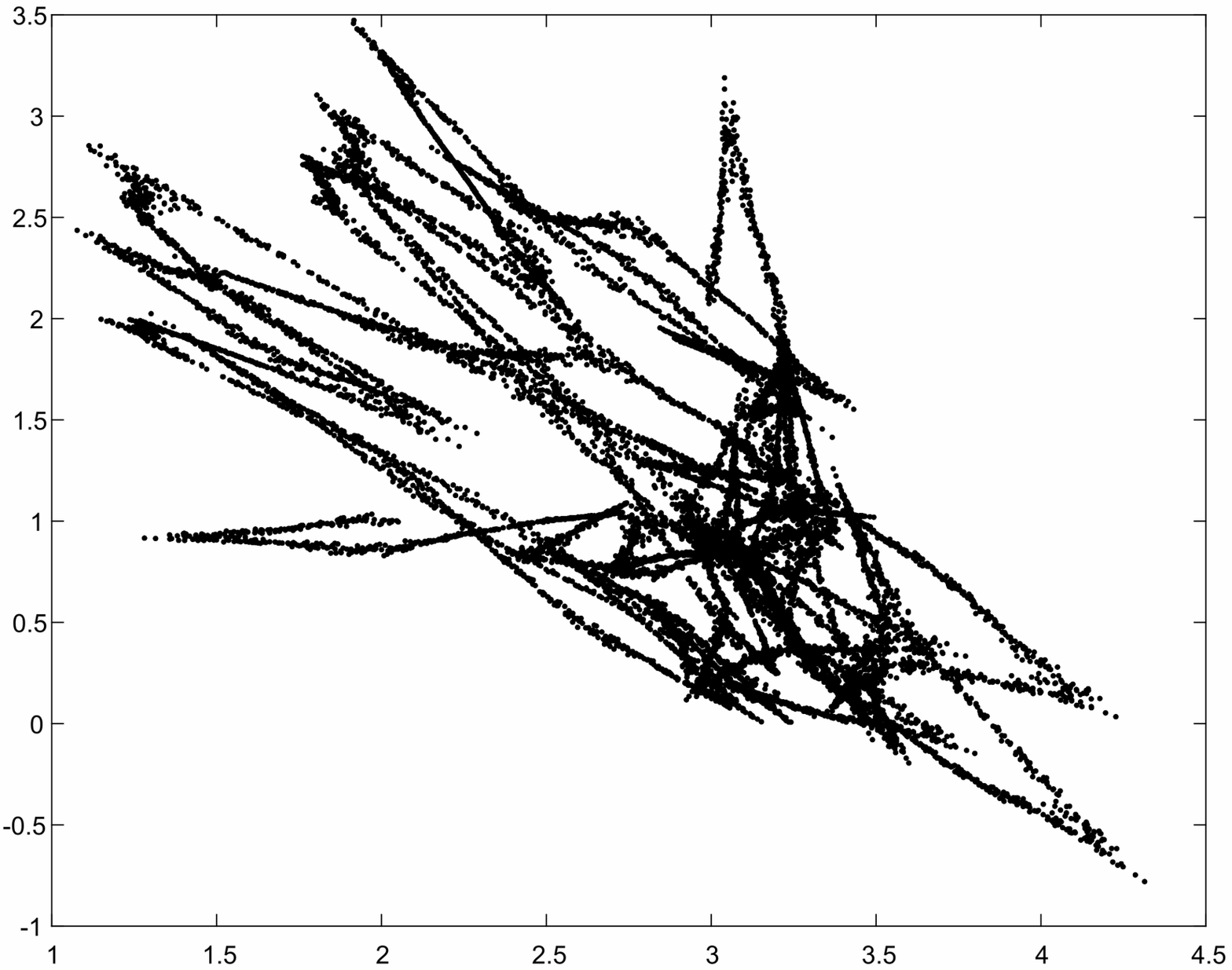}
    \caption{$w_k\in [-0.8,0.8]$; \ \ Backward trajectories:\ \ $\Psi_k(A)$, $k=10,12,14$.}
    \label{randw3}
\end{figure}
%\fi

\end{document}